\documentclass[11pt,reqno]{amsart}
\usepackage[T1]{fontenc}
\usepackage{graphicx}
\usepackage{cite}

\usepackage{color}
\definecolor{MyLinkColor}{rgb}{0,0,0.4}
\newcommand{\R}{{\mathbb R}}

\newcommand{\bA}{{\mathbb A}}
\newcommand{\bB}{\mathbb{B}}
\newcommand{\Z}{{\mathbb Z}}
\newcommand{\N}{{\mathbb N}}

\newcommand{\kH}{\mathcal{H}}

\newcommand{\kL}{\mathcal{L}}

\newcommand{\wt}{\widetilde}
\newcommand{\oo}{\overline\omega}
\newcommand{\re}{\mathop{\rm Re}\nolimits}

\newcommand{\PV}{\mathop{\rm PV}\nolimits}
\newcommand{\ov}{\overline}
\newcommand{\p}{\partial}

\newcommand{\e}{\varepsilon}

\newcommand{\id}{\mathop{\rm id}\nolimits}

\newcommand{\supp}{\mathop{\rm supp}\nolimits}
\newtheorem{thm}{Theorem}[section]

\newtheorem{lemma}[thm]{Lemma}

\theoremstyle{remark}

\numberwithin{equation}{section} 

\title[The Muskat problem with surface tension is subcritical  $L_p$-Sobolev spaces]{The Muskat problem with surface tension and equal viscosities in subcritical  $L_p$-Sobolev spaces}

\author{Anca-Voichita Matioc}
\address{Fakult\"at f\"ur Mathematik, Universit\"at Regensburg,   93040 Regensburg, Deutschland.}
\email{anca.matioc@ur.de}
\email{bogdan.matioc@ur.de}

\author{Bogdan--Vasile Matioc}


\usepackage[colorlinks=true,linkcolor=MyLinkColor,citecolor=MyLinkColor]{hyperref} %dvips

\begin{document}

\begin{abstract}
In this paper we establish the well-posedness of the Muskat problem with surface tension and equal viscosities in the subcritical  Sobolev spaces $W^s_p(\mathbb{R})$, where ${p\in(1,2]}$ and~${s\in(1+1/p,2)}$.
This is achieved by showing that the mathematical model can be formulated as a quasilinear parabolic evolution problem  in  $W^{\overline{s}-2}_p(\mathbb{R})$, where~${\overline{s}\in(1+1/p,s)}$.  
Moreover, we prove that the solutions become instantly smooth and we provide a criterion for the global existence of solutions.
\end{abstract}
\maketitle

%%%%%%%%%%%%%%%%%%%%%%%%%%%%%%%%%%%%%%%%%%%%%%%%%%
%%%%%%%%%%%%%%%%%%%%%%%%%%%%%%%%%%%%%%%%%%%%%%%%%%
%%%%%%%%%%%%%%%%%%%%%%%%%%%%%%%%%%%%%%%%%%%%%%%%%%
%%%%%%%%%%%%%%%%%%%%%%%%%%%%%%%%%%%%%%%%%%%%%%%%%%
\section{Introduction}\label {Sec:0}
%%%%%%%%%%%%%%%%%%%%%%%%%%%%%%%%%%%%%%%%%%%%%%%%%%
%%%%%%%%%%%%%%%%%%%%%%%%%%%%%%%%%%%%%%%%%%%%%%%%%%
%%%%%%%%%%%%%%%%%%%%%%%%%%%%%%%%%%%%%%%%%%%%%%%%%%
%%%%%%%%%%%%%%%%%%%%%%%%%%%%%%%%%%%%%%%%%%%%%%%%%%

In this paper we study the evolution equation
\begin{subequations}\label{P}
\begin{equation}\label{P:1}
\p_tf(t,x)=\displaystyle\frac{k}{2\pi\mu}\PV\int_\R\frac{y+\p_xf(t,x)\delta_{[x,y]}f(t)}{ y^2+(\delta_{[x,y]}f(t))^2 }\p_x(\sigma\kappa(f)-\Delta_\rho f)(t,x-y)\, dy,
\end{equation}
which is defined for $ t>0$ and  $x\in\R$.
The function  $f$  is assumed to be known at time $t=0$, that is 
\begin{equation}\label{P:2}
f(0,\cdot)= f_0.
\end{equation}
\end{subequations}
The evolution problem \eqref{P} is the contour integral formulation of the Muskat problem with surface tension and with/without gravity effects, see \cite{MBV19, MBV18} for an equivalence proof
of~\eqref{P} to the classical formulation of the Muskat problem \cite{Mu34}.
The problem~\eqref{P} describes the two-dimensional motion of two fluids with equal viscosities $\mu_-=\mu_+=\mu$ and general densities $\rho_-$ and $\rho_+$ in a vertical/horizontal 
homogeneous porous medium which is identified with $\R^2$. 
The  fluids occupy the entire plane, they are separated by the sharp interface~${\{y=f(t,x)+Vt\}}$, and  they move with constant velocity~${(0,V)}$, where~${V\in\R}$.
The fluid denoted by $+$ is located above this moving interface.
We use~${g\in[0,\infty)}$ to denote  the Earth's gravity, $k>0$ is  the permeability of the homogeneous porous medium, and
$\sigma>0$ is the surface tension coefficient at the free boundary.
Moreover, to shorten the notation we have set $\Delta_\rho:=g(\rho_--\rho_+)\in\R$
and 
\[
\delta_{[x,y]} f:=f(x)-f(x-y),\qquad x,\, y\in\R.
\]
Finally, $\kappa(f(t))$ is the curvature of $\{y=f(t,x)+tV\}$ and
$\PV$ denotes the principal value.

The Muskat problem with surface tension has received much interest in the recent years. 
Besides the fundamental  well-posedness issue also other important important features like the stability of stationary solutions \cite{EEM09c, MBV20, EM11a,  EMM12a,  PS16x, PS16, FT03, MW20},
 parabolic smoothing properties~\cite{MBV19, MBV18, MBV20}, 
 the zero surface tension limit \cite{A14, NF20x}, and the degenerate limit when the thickness of the fluid layers (or a certain nondimensional parameter) vanishes \cite{EMM12, GS19}
 have been investigated in this context.
  We also refer to \cite{EsSi97, GGL20} for results on the Hele-Shaw problem with surface tension effects, which is the one-phase version of the Muskat problem,
 and to \cite{Tao97, PrS18, PSW19} for results on the related Verigin problem with surface tension.
 
Concerning the well-posedness of the Muskat problem with surface tension effects, this property has been 
investigated in bounded (layered) geometries in \cite{EM11a,  EMM12a,  EMW18, PS16x, PS16} where
   abstract parabolic theories have been employed in the analysis,
 the approach in \cite{HTY97} relies on Schauder's fixed-point theorem, and in \cite{BV11}
  the authors use  Schauder's fixed-point theorem in a setting which allows for a sharp  corner point of the initial geometry. 

The results on the Muskat problem with surface tension in the unbounded geometry considered in this paper (and possibly in the general case of  fluids with different viscosities) are more recent,
 cf. \cite{A14, To17, MBV19, MBV18, MBV20, Ngu20, NF20x}.
While in   \cite{A14, To17}  the initial  data are chosen from~${H^s(\mathbb{T})}$, with $s\geq 6$, the regularity of the initial data has been decreased in \cite{MBV19, MBV18, MBV20} to $H^{2+\e}(\R)$, with~$\e\in(0,1)$ 
arbitrarily small. 
Finally, the very recent references \cite{Ngu20, NF20x} consider the problem  with initial data in~$H^{1+\frac{d}{2}+\e}(\R^d)$, with $d\geq1$ and $\e>0$ arbitrarily small, covering  
all subcritical $L_2$-based Sobolev spaces in all dimensions.

It is the aim of this paper to study the Muskat problem~\eqref{P}  in the subcritical $L_p$-based Sobolev spaces  $W^s_p(\R)$ with $p\in(1,2]$ and $s\in(1+1/p,2)$.
This issue is new in the context of \eqref{P} (see \cite{AbMa20x, CGSV17} for results in the case when $\sigma=0$).
To motivate why $W^{1+1/p}_p(\R)$ is a critical space for \eqref{P} we first emphasize that the surface tension  is the dominant factor for the dynamics as it contains the highest spatial derivatives of $f$.
Besides, if we set $g=0$, then it is not difficult to show that if $f$ is a solution to \eqref{P:1}, then, given $\lambda>0$, the function  
\[
f_\lambda(t,x):=\lambda^{-1}f(\lambda^3  t,\lambda x)  
\]  
 also solves \eqref{P:1}. 
This scaling identifies  $W^{1+1/p}_p(\R)$ as a critical space for~\eqref{P}.
 The main result Theorem~\ref{MT1} establishes the well-posedness of \eqref{P} in  $W^s_p(\R)$.
 This is achieved by showing that \eqref{P} can be recast as a quasilinear parabolic evolution equation, so that abstract results available
  for this type of problems, cf. \cite{Am88, Am86, Am93, MW20}, can be applied in our context. 
  A particular feature of the Muskat problem \eqref{P} is the fact that the equations have to be interpreted in distributional 
  sense as they are realized in the Sobolev space~$W^{\ov s-2}_p(\R)$, where~$\ov s$ is chosen such that $1+1/p<\ov s<s<2$.
Additionally to well-posedness,  Theorem~\ref{MT1} provides two parabolic smoothing properties, showing in particular that~\eqref{P:1}  holds pointwise,
and a criterion  for the global existence of solutions.

\begin{thm}\label{MT1}
Let  $p\in(1,2]$ and $1+1/p<\ov s<s<2$.
Then, the Muskat  problem \eqref{P} possesses  for each   $f_0\in W^s_p(\R)$ a unique maximal solution $f:=f(\,\cdot\,; f_0)$ such that
\[ f\in {\rm C}([0,T^+), W^s_p(\R))\cap {\rm C}((0,T^+ ),  W^{\ov s+1}_p(\R))\cap {\rm C}^1((0,T^+),  W^{\ov s-2}_p(\R)), \] 
with $T^+=T^+(f_0)\in(0,\infty]$ denoting the maximal time of existence.
Moreover, the following properties hold true:
\begin{itemize}
\item[(i)] The solution depends continuously on the initial data;
\item[(ii)] Given $k\in\N$, we have $f\in {\rm C}^\infty((0,T^+ )\times\R,\R)\cap {\rm C}^\infty ((0,T^+),  W^k_p(\R));$
 \item[(iii)] The solution is global if 
\[
\sup_{[0,T^+)\cap[0,T]}\|f(t)\|_{ W^s_p(\R)}<\infty\qquad\text{for all $T>0$}.
\]
\end{itemize}
 \end{thm}
 
We emphasize that some of the  arguments   use in an essential way the fact that $p\in(1,2]$.
More precisely, we employ several times   the Sobolev embedding ${W^s_p(\R)\hookrightarrow W^{t}_{p'}(\R)},$
where $p'$ is the   adjoint exponent of $p$, that is $p^{-1}+{p'}^{-1}=1$ (this notation is used in the entire paper).
 This  provides the restriction $p\in(1,2]$.

Additionally, we expect that the assertion (ii) of Theorem~\ref{MT1} can be improved to  real-analyticity instead of smoothness. 
However,  this would require to establish real-analytic dependence of the right-hand side of \eqref{P:1} on $f$ in the functional analytic framework considered in Section~\ref{Sec:2}, 
which is much more involved than   showing the smooth dependence (see \cite[Proposition 5.1]{MBV19} for a related proof of real-analyticity).

\subsection*{Notation}
 Given $k\in\N$ and $p\in(1,\infty),$ we let $W^k_p(\R)$ denote the standard $L_p$-based Sobolev space with norm  
\[
\|f\|_{W^k_p}:=\Big(\sum_{\ell=0}^k\|f^{(\ell)}\|_p^p\Big)^{1/p}.
\]
 Given $0<s\not\in\N$ with $s=[s]+\{s\}$, where~${[s]\in\N}$ and  $\{s\}\in(0,1)$, the Sobolev space $W^s_p(\R)$ is the subspace of $W^{[s]}_p(\R)$ that consists of functions for which the seminorm
\[
[f]_{W^{s}_p}^p:=\int_{\R^2}\frac{|f^{([s])}(x)-f^{([s])}(y)|^p}{|x-y|^{1+\{s\}p}}\, d(x,y)=\int_{\R}\frac{\|f^{([s])} -\tau_\xi f^{([s])}\|_p^p}{|\xi|^{1+\{s\}p}}\, d\xi
\] 
is finite.
Here $\{\tau_\xi\}_{\xi\in\R}$ is the group of right translations and  $\|\cdot\|_q:=\|\cdot\|_{L_q(\R)}$,~$q\in[1,\infty].$
The norm on $W^s_p(\R)$ is defined by
\[
\|f\|_{W^s_p}:=\big(\|f\|_{W^{[s]}_p}^p+[f]_{W^{s}_p}^p\big)^{1/p}.
\]
For $s<0$, $W^s_p(\R)$ is defined as the dual space of $W^{-s}_{p'}(\R)$.
 
 The following properties can be found e.g. in \cite{Tr83}.
\begin{itemize}
\item[(i)] ${\rm C}^\infty_0(\R)$ lies dense in $W^s_p(\R)$ for all $s\in\R$. Moreover,  $W^s_p(\R)\hookrightarrow {\rm C}^{s-1/p}(\R)$  holds provided that $0<s-1/p\not\in\N$.
\item[(ii)] $W^s_p(\R)$ is an algebra for $s>1/p$.
\item[(iii)] If $\rho>\max\{s,  -s\},$ then $f\in{\rm C}^\rho(\R)$ is a pointwise  multiplier for $W^s_p(\R)$, that is
\begin{align}\label{adMP}
\|fg\|_{W^s_p}\leq C\|f\|_{{\rm C}^\rho}\| g\|_{W^s_p} \qquad \text{for all $g\in W^s_p(\R)$},
\end{align}
with $C$ independent of $f$ and $g$.
\item[(iv)] Given $\theta\in(0,1)$ and $p\in(1,\infty)$, let $(\cdot,\cdot)_{\theta,p}$ denote the real interpolation functor of exponent $\theta$ and parameter $p\in(1,\infty).$
Given $s_0,\, s_1\in\R$ with $(1-\theta)s_0+\theta s_1\not\in\Z$, it holds that
\begin{align}\label{IP}
(W^{s_0}_{p}(\R), W^{s_1}_{p}(\R))_{\theta,p}=W^{(1-\theta)s_0+\theta s_1}_{p}(\R).
\end{align}
\item[(v)]  $ W^s_p(\R)\hookrightarrow W^{t}_q(\R))$ if $1<p\leq q<\infty$ and $ s-1/p\geq t-1/q.$
\end{itemize} \medskip

Besides, we need also the following properties.
\begin{itemize}
\item[(a)]  Given $r\in[0,1)$ and $p\in(1,\infty)$,  there exists a constant $C>0$ such that
\begin{align}\label{MES}
 \|gh\|_{W^{r}_p}\leq  2\|g\|_\infty\|h\|_{W^{r}_p}+C\|g\|_{W^{r+1}_p}\|h\|_p,\quad\text{$g\in W^{r+1}_p(\R)$, $h\in W^{r}_p(\R)$.}
\end{align} 
\item[(b)] Given $r\in(1/p,1) $ and $p\in(1,\infty)$,  there exists a constant $C>0$ such that 
 \begin{align}\label{MES2}
 \|gh\|_{W^{r}_p}\leq  2(\|g\|_\infty\|h\|_{W^{r}_p}+\|h\|_\infty\|g\|_{W^{r}_p}),\quad\text{$g,\, h\in W^{r}_p(\R)$.}
\end{align}
\item[(c)] Given $p\in(1,2]$, $r\in(1/p,1)$, and $\rho\in(0,\min\{r-1/p,1-r\}),$ there exists a constant~$C>0$ such that 
\begin{align}\label{EQ:Mul1}
\|gh\|_{W^{1-r}_{p'}}\leq 2\|g\|_\infty \|h\|_{W^{1-r}_{p'}}+C\|g\|_{W^{r}_p}\|h\|_{W^{1-r-\rho}_{p'}} 
\end{align}
for all $g\in W^{r}_p(\R)$ and $h\in W^{1-r}_{p'}(\R),$ and
\begin{align}\label{EQ:Mul2}
\|\varphi h\|_{W^{r-1}_p}\leq 5\|\varphi\|_\infty \|h\|_{W^{r-1}_p}+C\frac{ \|\varphi\|_{W^{r}_{p}}^{1+2r/\rho}}{\|\varphi\|_{\infty}^{2r/\rho}}\|h\|_{W^{r-1-\rho}_p} 
\end{align}
for all $h\in W^{r-1}_p(\R) $ and  $0\neq \varphi\in W^r_p(\R)$.
\end{itemize}
The estimates \eqref{MES} and \eqref{MES2} are straightforward consequences of the properties (i)-(v) listed above.
The inequalities~\eqref{EQ:Mul1} and~\eqref{EQ:Mul2} are established in Appendix~\ref{Sec:b}.
Let us point out that~\eqref{EQ:Mul1} implies in particular that the multiplications
\begin{equation}\label{mult}
\begin{aligned}
&[(g,h)\mapsto gh]: W^{r}_{p}(\R)\times W^{1-r}_{p'}(\R)\to W^{1-r}_{p'}(\R),\\[1ex] 
& [(g,h)\mapsto gh]: W^{r}_{p}(\R)\times W^{r-1}_{p}(\R)\to W^{r-1}_{p}(\R) 
\end{aligned}
\end{equation}
are both  continuous if $p\in(1,2]$ and $r\in(1/p,1)$.
The continuity of $\eqref {mult}_1$ is a straightforward consequence of \eqref{EQ:Mul1}, while  the continuity of $\eqref {mult}_2$ follows by a standard duality argument.
\subsection*{Outline} In Section~\ref{Sec:1} we introduce some multilinear singular integral operators and study their properties. 
These operators are then used in Section~\ref{Sec:2} to formulate~\eqref{P} as a quasilinear evolution problem, cf.~\eqref{AF}-\eqref{PHI}. 
Subsequently, we show in Theorem~\ref{T:GP} that~\eqref{AF} is of parabolic type and we complete the section with the proof of Theorem~\ref{MT1}.
In Appendix~\ref{Sec:A} and Appendix~\ref{Sec:b} we prove some technical results used in the analysis.

%%%%%%%%%%%%%%%%%%%%%%%%%%%%%%%%%%%%%%%%%%%%%%%%%%
%%%%%%%%%%%%%%%%%%%%%%%%%%%%%%%%%%%%%%%%%%%%%%%%%%
%%%%%%%%%%%%%%%%%%%%%%%%%%%%%%%%%%%%%%%%%%%%%%%%%%
%%%%%%%%%%%%%%%%%%%%%%%%%%%%%%%%%%%%%%%%%%%%%%%%%%
\section{Some singular integral operators}\label {Sec:1}
%%%%%%%%%%%%%%%%%%%%%%%%%%%%%%%%%%%%%%%%%%%%%%%%%%
%%%%%%%%%%%%%%%%%%%%%%%%%%%%%%%%%%%%%%%%%%%%%%%%%%
%%%%%%%%%%%%%%%%%%%%%%%%%%%%%%%%%%%%%%%%%%%%%%%%%%
%%%%%%%%%%%%%%%%%%%%%%%%%%%%%%%%%%%%%%%%%%%%%%%%%%

In this section we investigate a family of multilinear singular integral operators which play a key role in the analysis of the Muskat problem (and also of the Stokes problem~\cite{MP20x}).
Given~${n,\,m\in\N}$ and  Lipschitz continuous  functions ${a_1,\ldots, a_{m},\, b_1, \ldots, b_n:\mathbb{R}\to\mathbb{R}}$  we set
\begin{equation}\label{BNM}
B_{n,m}(a_1,\ldots, a_m)[b_1,\ldots,b_n,\oo](x):=\PV\int_\mathbb{R}  \frac{\oo(x-y)}{y}\cfrac{\prod_{i=1}^{n}\big(\delta_{[x,y]} b_i /y\big)}{\prod_{i=1}^{m}\big[1+\big(\delta_{[x,y]}  a_i /y\big)^2\big]}\, dy.
\end{equation}
For brevity  we  write
\begin{equation}\label{defB0}
B^0_{n,m}(f)[\oo]:=B_{n,m}(f,\ldots f)[f,\ldots,f,\oo].
\end{equation}
The relevance of these operators in the context of~\eqref{P}  is enlightened by the fact that~\eqref{P:1}  can be recast, at least at formal level, in a compact form as
\begin{equation}\label{P'}
  \cfrac{d f(t)}{dt}=\frac{k}{2\mu}\bB(f(t))[(\sigma\kappa(f(t))-\Delta_\rho f(t))'],\quad t>0,\\[2ex]
\end{equation}
where~$\bB(f)$ is  defined by
\begin{align}
\bB(f)&:= \pi^{-1}(B_{0,1}^0(f)+f'B_{1,1}^0(f)) \label{OpB} 
\end{align}
and $f':=df/dx$.
A key observation that we exploit in our analysis is the quasilinear structure of the curvature operator.
Indeed, it holds that
\[
\kappa(f)=\kappa(f)[f],
\]
where
\begin{align}\label{qlk}
\kappa(f)[h]:=\frac{h''}{(1+f'^2)^{3/2}}.
\end{align}
Given $f\in W^s_p(\R)$ and $h\in W^{s+1}_p(\R)$, with $p\in(1,\infty)$ and $s\in(1+1/p,2),$   Lemma~\ref{L:1} ensures that the term
\[
(\sigma\kappa(f)[h]-\Delta_\rho h)',
\]
which corresponds to the linear variable in \eqref{P'}, belongs to $W^{s-2}_p(\R)$.
Therefore, it is natural  to ask weather $\bB(f)\in\kL(W^{s-2}_p(\R))$.  
The proof of this boundedness property, which enables us to view \eqref{P'} as an evolution equation in~${W^{s-2}_p(\R)}$, see Section~\ref{Sec:2}, is the main goal of this section.
As already mentioned, some of the arguments require     $p\in(1,2]$.
 \medskip

We first recall the following result.

\begin{lemma}\label{L:MP0} Let $p\in(1,\infty)$, $n,\,m \in\N$, and let $a_1,\ldots, a_{m},\, b_1, \ldots, b_n:\R\to\R$ be Lipschitz continuous.
Then, there exists a constant  $C=C(n,\, m,\,\max_{i=1,\ldots, m}\|a_i'\|_{\infty} )$
such that
\[
\|B_{n,m}(a_1,\ldots, a_m)[b_1,\ldots,b_n,\,\cdot\,]\|_{\kL(L_p(\R))}\leq C\prod_{i=1}^{n} \|b_i'\|_{\infty}.
\]
 Moreover\footnote{Given $ n\in\N$ and Banach spaces $X$ and $Y$, we let $\kL^n_{\rm sym}(X,Y)$ denote the space of $n$-linear, bounded symmetric maps $A:\;X^n\longrightarrow Y$.},   $B_{n,m}\in {\rm C}^{1-}((W^1_\infty(\R))^{m},\kL^n_{\rm sym}(W^1_\infty(\R),\kL( L_p(\R))).$\\[-2ex]
\end{lemma}
\begin{proof}
See \cite[Lemma 2.2]{AbMa20x}.
\end{proof}

We point out that, given     Lipschitz continuous functions $ a_1,\ldots,  a_m, \,\wt a_1,\ldots, \wt a_m,\, b_1, \ldots, b_n$, we have 
\begin{equation}\label{spr3}   
   \begin{aligned}
 &B_{n,m}(\wt a_{1}, \ldots, \wt a_{m})[b_1,\ldots, b_{n},\cdot]-  B_{n,m}(a_1, \ldots, a_{m})[b_1,\ldots, b_{n},\cdot]\\[1ex]
 &\hspace{0.5cm}=\sum_{i=1}^{m} B_{n+2,m+1}(\wt a_{1},\ldots, \wt a_{i},a_i,\ldots \ldots, a_{m})[b_1,\ldots, b_{n},a_i+\wt a_{i}, a_i-\wt a_{i},\cdot].
\end{aligned}
\end{equation}
The formula \eqref{spr3} was used to establish the Lipschitz continuity property  (denoted by   ${\rm C}^{1-}$) in Lemma~\ref{L:MP0} and is  also of importance for our later analysis.

The strategy is as follows. 
In Lemma~\ref{L:MP2} we show  that, given $f\in W^s_p(\R)$, with $p\in(1,2]$ and $s\in(1+1/p,2)$, we have $B_{n,m}^0(f)\in \kL(W^{r}_{p'}(\R)) $ for all $r\in[0,1-1/p)$ (in particular also for $r=2-s$).
Lemma~\ref{L:MP1} below provides the key argument in the proof of Lemma~\ref{L:MP2}. 
The desired mapping property $B_{n,m}^0(f)\in \kL(W^{s-2}_p(\R))$ stated in Lemma~\ref{L:MP3},  follows then from Lemma~\ref{L:MP2} via a duality argument.
Lemma~\ref{L:MP3} and the fact that $f'\in W^{s-1}_p(\R)$ is a pointwise multiplier for $W^{s-2}_p(\R)$, see $\eqref{mult}$, provide the desired property ${\bB(f)\in\kL(W^{s-2}_p(\R))}$.

\begin{lemma}\label{L:MP1}
Given    $p\in(1,2]$, $s\in(1 +1/p,2)$, $r\in(0, 1-1/p)$,  
 $n,\, m\in\N $, $n\geq1$, and   ${a_1,\ldots, a_m \in W^s_p(\R)}$, 
there exists a constant~$C=C(n,\, m,\, s,\,r,\,p,\,\max_{1\leq i\leq m}\|a_i\|_{W^s_p})$  such that
\begin{align} 
\| B_{n,m}(a_1,\ldots, a_{m})[b_1,\ldots, b_n,\oo]\|_{p'}\leq C \|b_1 \|_{W^{s+1-r-2/p}_{p'}}\|\oo\|_{W^{r}_{p'}}\prod_{i=2}^{n}\|b_i\|_{W^s_p} \label{REF1}
\end{align}
for all $ b_1,\ldots, b_n\in W^s_p(\R)$ and $\oo\in W^{r}_{p'}(\R).$
\end{lemma}
\begin{proof}
Without loss of generality  we may assume that $\oo\in {\rm C}^\infty_0(\R).$ 
Using the relation
\[
\frac{\p}{\p y}\Big(\frac{\delta_{[x,y]} b_1}{y}\Big)=\frac{ b_1'(x-y)}{y}- \frac{\delta_{[x,y]} b_1}{y^2},
\]
algebraic manipulations lead us to
\begin{align*}
&\hspace{-0.5cm}B_{n,m}(a_1,\ldots, a_{m})[b_1,\ldots, b_n,\oo](x)\\[1ex]
&=\big(\oo B_{n-1,m}(a_1,\ldots, a_{m})[b_2,\ldots, b_n,b_1']\big)(x)+\int_\R K(x,y)\, dy\\[1ex]
&\hspace{0,45cm}-\oo(x)\int_\R \cfrac{\prod_{i=2 }^{n}\big(\delta_{[x,y]} b_i /y\big)}{\prod_{i=1}^{m}\big[1+\big(\delta_{[x,y]}  a_i /y\big)^2\big]}  \frac{\p}{\p y}\Big(\frac{\delta_{[x,y]} b_1 }{y}\Big)\, dy
\end{align*}
for $x\in\R$, where 
\[
K(x,y):=-\cfrac{\prod_{i=1 }^{n}\big(\delta_{[x,y]} b_i /y\big)}{\prod_{i=1}^{m}\big[1+\big(\delta_{[x,y]}  a_i /y\big)^2\big]}  \frac{\delta_{[x,y]} \oo }{y},\qquad\text{$x\in\R$, $y\neq 0$.}
\]
The last term on the right side of the previous identity vanishes if ${(n-1)^2+m^2=0}$. 
Otherwise, we use integration  by parts and arrive at 
\begin{equation}\label{RB}
\begin{aligned}
&\hspace{-0.5cm}B_{n,m}(a_1,\ldots, a_{m})[b_1,\ldots, b_n,\oo](x)\\[1ex]
&=\big(\oo B_{n-1,m}(a_1,\ldots, a_{m})[b_2,\ldots, b_n,b_1']\big)(x)+\int_\R K(x,y)\, dy\\[1ex]
&\hspace{0.445cm}-\oo(x)\sum_{j=2}^n\int_\R  K_{1,j}(x,y)\, dy+\oo(x)\sum_{j=1}^m\int_\R K_{2,j}(x,y)\, dy,
\end{aligned}
\end{equation}
where, given $x\in\R$ and $y\neq 0$, we set
\begin{align*}
 K_{1,j}(x,y)&:= \cfrac{\prod_{i=1, i\neq j }^{n}\big(\delta_{[x,y]} b_i /y\big)}{\prod_{i=1}^{m}\big[1+\big(\delta_{[x,y]}  a_i /y\big)^2\big]}\frac{\delta_{[x,y]}b_j-yb_j'(x-y)}{y^2},\\[1ex]
 K_{2,j}(x,y)&:= 2\cfrac{ \delta_{[x,y]} a_j /y}{1+\big(\delta_{[x,y]}  a_j /y\big)^2}
 \cfrac{\prod_{i=1 }^{n}\big(\delta_{[x,y]} b_i /y\big)}{\prod_{i=1}^{m}\big[1+\big(\delta_{[x,y]}  a_i /y\big)^2\big]}
  \frac{\delta_{[x,y]}a_j-ya_j'(x-y)}{y^2}.
\end{align*}
We estimate the $L_{p'}$-norm of the four terms on the right of \eqref{RB} separately.
 Let $q\in( p',\infty)$ be defined as the solution to
\[
\frac{1}{q}+\frac{1}{1/r}=\frac{1}{p'}.
\]

\noindent{\em Term 1.} We note that ${\oo\in W^{r}_{p'}(\R)\hookrightarrow L_q(\R)}$,
  ${b_1'\in W^{s-2/p}_{p'}(\R)\hookrightarrow W^{s-r-2/p}_{p'}(\R)}$, and additionally we have~$W^{s-r-2/p}_{p'}(\R)\hookrightarrow L_{1/r}(\R) $.
Hölder's inequality together with Lemma~\ref{L:MP0} (with $p=1/r$) then yields
\begin{equation}\label{DE0}
\begin{aligned}
\|\oo B_{n-1,m}(a_1,\ldots, a_{m})[b_2,\ldots, b_n,b_1']\|_{p'}&\leq \|\oo\|_q\|B_{n-1,m}(a_1,\ldots, a_{m})[b_2,\ldots, b_n,b_1']\|_{1/r}\\[-1ex]
&\leq  C\|\oo\|_{q}\|b_1'\|_{1/r}\Big(\prod_{i=2}^{n}\|b_i'\|_{\infty}\Big)\\[-1ex]
&\leq C\|\oo\|_{W^{r}_{p'}}\|b_1 \|_{W^{s+1-r-2/p}_{p'}}\Big(\prod_{i=2}^{n}\|b_i\|_{W^s_p}\Big).
\end{aligned}
\end{equation}

\noindent{\em Term 2.}
Let $s_0\in (1+1/p, s]$ be chosen such that $s_0-r-1/p<1$.  Minkowski's integral inequality, H\"older's inequality,  and  the property $b_1\in  W^{s+1-r-2/p}_{p'}(\R)\hookrightarrow {\rm C}^{s_0-r-1/p}(\R)$ lead~to
\begin{align*}
\Big(\int_\R\Big|\int_\R K(x,y)\, dy  \Big|^{p'}\, dx\Big)^{1/p'} &\leq\Big(\int_{\{|y|<1\}}+\int_{\{|y|>1\}}\Big)\Big(\int_\R |K(x,y)|^{p'}\, dx  \Big)^{1/p'}\, dy\\[1ex]
 &\leq 4\|b_1\|_\infty \Big(\prod_{i=2}^n \|b_i'\|_\infty\Big)\|\oo\|_{p'}\\[1ex]
&\hspace{0,45cm}  +  [b_1]_{ s_0-r-1/p}\Big(\prod_{i=2}^n \|b_i'\|_\infty\Big)\int_{\{|y|<1\}}\frac{\|\oo-\tau_y\oo\|_{p'}}{|y|^{2- s_0+r+1/p}}\, dy,
\end{align*}
where
\begin{align*}
\int_{\{|y|<1\}}\frac{\|\oo-\tau_y\oo\|_{p'}}{|y|^{2- s_0+r+1/p}}\, dy &\leq  \|\oo\|_{W^{r}_{p'}}\Big(\int_{\{|y|<1\}}\frac{1}{|y|^{(1-s_0+2/p)p}}\, dy\Big)^{1/p}\leq C\|\oo\|_{W^{r}_{p'}}.
\end{align*}
We arrive at
\begin{align}\label{DE:4}
\Big(\int_\R\Big|\int_\R K(x,y)\, dy  \Big|^{p'}\, dx\Big)^{1/p'} \leq C\|\oo\|_{W^{r}_{p'}} \|b_1\|_{W^{s+1-r-2/p}_{p'}}  \Big(\prod_{i=2}^{n}\|b_i\|_{W^s_p}\Big).
\end{align}

\noindent{\em Terms 3 $\&$ 4.} Given $2\leq j\leq n$,    Hölder inequality  and Minkowski's  integral inequality yield
\begin{equation*}
\begin{aligned}
 \Big(\int_\R\Big|\oo(x)\int_\R K_{1,j}(x,y)\, dy\Big|^{p'}\, dx\Big)^{1/p'}&\leq\|\oo\|_q \Big(\int_\R\Big|\int_\R K_{1,j}(x,y)\, dy\Big|^{\frac{1}{r}}\, dx\Big)^{r}\\[1ex]
&\leq \|\oo\|_{q}\int_\R\Big(\int_\R |K_{1,j}(x,y)|^{\frac{1}{r}}\, dx\Big)^{r}\, dy.
\end{aligned}
\end{equation*} 
To estimate the integral term we choose again  $s_0\in (1+1/p, s]$ such that $s_0-r-1/p<1$. 
Taking into account that $b_1\in  W^{s+1-r-2/p}_{p'}(\R)\hookrightarrow W^1_{1/r}(\R)\cap {\rm C}^{s_0-r-1/p}(\R)$, we get
\begin{align*}
&\hspace{-0.25cm}\int_\R\Big(\int_\R |K_{1,j}(x,y)|^{1/r}\, dx\Big)^{r}\, dy\\[1ex]
&\leq 4\|b_1\|_{1/r}\Big(\prod_{i=2 }^{n}\|b_i\|_{W^{s}_p}\Big)+[b_1]_{s_0-r-1/p}\Big(\prod_{i=2, i\neq j}^{n}\|b_i'\|_{\infty}\Big)
\int_{\{|y|<1\}}\frac{\|b_j-\tau_y b_j-y\tau_y b_j'\|_{1/r}}{|y|^{3-s_0+r+1/p}} \, dy 
\end{align*}
for  $2\leq j\leq n.$ 
Since  $b_j\in W^{s_0}_p(\R)\hookrightarrow W^{s_0+r-\frac{1}{p} }_{1/r}(\R)$, $2\leq j\leq n $,   Minkowski's integral inequality, H\"older's inequality, and a change of variables lead to 
\begin{align*}
&\hspace{-0.5cm}\int\limits_{\{|y|<1\}}\frac{\|b_j-\tau_y b_j-y\tau_y b_j'\|_{1/r}}{|y|^{3-s_0+r+1/p}} \, dy \\[1ex]
&=\int_{\{|y|<1\}}\frac{1 }{|y|^{2- s_0+r+1/p}} 
\Big(\int_\R \Big|\int_0^1 [b_j'(x-(1-t )y)- b_j'(x-y)]\, dt\Big|^{1/r}\, dx\Big)^{r}\, dy\\[1ex]
  &\leq  \int_0^1 \int_{\{|y|<1\}}\frac{\|b_j'- \tau_{-t y}b_j'\|_{1/r}}{|y|^{2- s_0+r+1/p}}\,dy\, dt\\[1ex]
  &\leq   \Big(\int_{\{|y|<1\}}\frac{1 }{|y|^{(3-2s_0-r+2/p)/(1-r)}}\, dy\Big)^{1-r}\Big(\int_0^1 t^{1- s_0+r+1/p}\, dt\Big)\|b_j'\|_{W^{ s_0+r-1-\frac{1}{p} }_{1/r}}\\[1ex]
  &\leq  C\|b_j\|_{W^{s}_p}.
\end{align*}
Consequently, given  $2\leq j\leq n $,  we have
\begin{align}\label{DE:2}
\Big(\int_\R\Big|\oo(x)\int_\R K_{1,j}(x,y) \, dy\Big|^{p'}\, dx\Big)^{1/p'}\leq C\|\oo\|_{W^{r}_{p'}} \|b_1\|_{W^{ s+1-r-2/p}_{p'}} \Big(\prod_{i=2}^{n}\|b_i\|_{W^s_p}\Big)
\end{align} 
and, similarly,
\begin{align}\label{DE:3}
\Big(\int_\R\Big|\oo(x)\int_\R K_{2,j}(x,y)\, dy\Big|^{p'}\, dx\Big)^{1/p'}\leq    C\|\oo\|_{W^{r}_{p'}} \|b_1\|_{W^{ s+1-r-2/p}_{p'}} \Big(\prod_{i=2}^{n}\|b_i\|_{W^s_p}\Big).
\end{align}
The desired claim follows now from \eqref{RB}-\eqref{DE:3}.
\end{proof}

Lemma~\ref{L:AL0} below   is used to prove Lemma~\ref{L:MP2}   (see also \cite[Lemma 2.7]{AbMa20x} for a related result).
\begin{lemma}\label{L:AL0} Let $p\in(1,\infty)$, $n\in\N$, and $n<t'<t<n+1$.
 Then, there exists   $C>0$ such that
 \begin{align*}
\int_\R\frac{\|b-\tau_\xi b\|_{W^{t'}_p}^p}{|\xi|^{1+(t-t')p}}\, d\xi\leq C\|b\|_{W^{t}_p}^p\qquad\text{for all $b\in W^{t}_p(\R)$.}
 \end{align*}
 \end{lemma}
 \begin{proof}
We have 
\[
\|b-\tau_\xi b\|_{W^{t'}_p}^p=\sum_{k=0}^n\|b^{(k)}-\tau_\xi b^{(k)}\|_{p}^p+[b-\tau_\xi b]_{W^{t'}_p}^p
\]
and, given $k\in\{0,\ldots,n\}$, 
\begin{align*}
\int_\R\frac{\|b^{(k)}-\tau_\xi b^{(k)}\|_{p}^p}{|\xi|^{1+(t-t')p}}\, d\xi&=[b^{(k)}]_{W^{t-t'}_p}^p\leq C\|b\|_{W^{t}_p}^p.
\end{align*}
Moreover, since $\R^2=\ov{\{|\xi|<|h|\}}\cup \{|\xi|>|h|\}$
 and
\begin{align*}
\int_\R\frac{[b-\tau_\xi b]_{W^{t'}_p}^p}{|\xi|^{1+(t-t')p}}\, d\xi&=\int_{\R^2}\frac{\|b^{(n)}-\tau_\xi b^{(n)}-\tau_h(b^{(n)}-\tau_\xi b^{(n)})\|_{p}^p}{|\xi|^{1+(t-t')p}|h|^{1+(t'-n)p}}\, dh\, d\xi,
\end{align*} 
with 
\begin{align*}
&\hspace{-0.5cm}\int\limits_{\{|\xi|<|h|\}}\frac{\|b^{(n)}-\tau_\xi b^{(n)}-\tau_h(b^{(n)}-\tau_\xi b^{(n)})\|_{p}^p}{|\xi|^{1+(t-t')p}|h|^{1+(t'-n)p}}\, dh\, d\xi\\[1ex]
&\leq 2^{p}\int_\R\frac{\|b^{(n)}-\tau_\xi b^{(n)}\|_p^p}{|\xi|^{1+(t-t')p}}
\Big(\int_{\{|\xi|<|h|\}}\frac{1}{|h|^{1+(t'-n)p}}\, dh\Big)\, d\xi\leq C[b]^p_{W^t_p},\\[1ex]
&\hspace{-0.5cm}\int_{\{|h|<|\xi|\}}\frac{\|b^{(n)}-\tau_\xi b^{(n)}-\tau_h(b^{(n)}-\tau_\xi b^{(n)})\|_{p}^p}{|\xi|^{1+(t-t')p}|h|^{1+(t'-n)p}}\, dh\, d\xi\\[1ex]
&\leq 2^{p}\int_\R\frac{\|b^{(n)}-\tau_h b^{(n)}\|_p^p}{|h|^{1+(t'-n)p}}
\Big(\int_{\{|h|<|\xi|\}}\frac{1}{|\xi|^{1+(t-t')p}}\, d\xi\Big)\, dh\leq C[b]^p_{W^t_p},
\end{align*}
  the desired estimate is immediate.
 \end{proof}

 We are now in a position to  prove that $B_{n,m}^0(f)\in \kL(W^{r}_{p'}(\R)) $ for all  $r\in[0,1-1/p)$.
 
 \begin{lemma}\label{L:MP2}
Given $p\in(1,2]$,  $s\in(1 +1/p,2)$,  $n,\, m\in\N$,    ${a_1,\ldots, a_m \in W^s_p(\R)}$, and ${r\in[0,1-1/p)}$,
 there exists a constant~$C=C(n,\, m,\,s,\,r,\,p,\, \max_{1\leq i\leq m}\|a_i\|_{W^s_p})$ such that
\begin{align} 
\| B_{n,m}(a_1,\ldots, a_{m})[b_1,\ldots, b_n,\oo]\|_{W^{r}_{p'}}\leq C \|\oo\|_{W^{r}_{p'}}\prod_{i=1}^{n}\|b_i\|_{W^{s}_p} \label{REF1'2}
\end{align}
for all $ b_1,\ldots, b_n\in W^s_p(\R)$ and $\oo\in W^{r}_{p'}(\R).$

Moreover,   $  B_{n,m}\in {\rm C}^{1-}((W^s_p(\R))^m,\kL^{n}_{\rm sym}( W_p^{s}(\R) , \kL(W^{r}_{p'}(\R)))).$ 
 \end{lemma}
\begin{proof}
  Let  $B_{n,m}:=B_{n,m}(a_1,\ldots, a_{m})[b_1,\ldots, b_n,\cdot].$ 
Recalling Lemma \ref{L:MP0} (with $p=p'$), we get 
\begin{align*} 
\| B_{n,m}[\oo]\|_{p'}\leq C \|\oo\|_{p'}\prod_{i=1}^{n}\|b_i\|_{W^s_p},
\end{align*}
which proves \eqref{REF1'2} for $r=0$. 
Let now ${r\in(0,1-1/p)}$.
It   remains to estimate  the quantity
\begin{align*}
[B_{n,m}[\oo]]_{W^{r}_{p'}}^{p'}&= \int_{\R}\frac{\|B_{n,m}[\oo]-\tau_\xi B_{n,m}[\oo]\|_{p'}^{p'}}{|\xi|^{1+r p'}}\, d\xi.
\end{align*}
Taking advantage of \eqref{spr3}, we write  
\begin{align*}
 B_{n,m}[\oo]-\tau_\xi B_{n,m}[\oo] = T_1(\cdot,\xi)+T_2(\cdot,\xi)-T_3(\cdot,\xi),\quad  \xi\in\R,
\end{align*}
with
\begin{align*}
T_1(\cdot ,\xi)&:=B_{n,m}(a_1,\ldots, a_{m})[b_1,\ldots, b_n,\oo-\tau_\xi\oo],\\[1ex]
T_2(\cdot,\xi)&:=\sum_{i=1}^nB_{n,m}(a_1,\ldots, a_{m})[\tau_\xi b_1,\ldots,\tau_\xi b_{i-1}, b_i-\tau_\xi b_i, b_{i+1},\ldots b_n,\tau_\xi\oo],\\[1ex]
T_3(\cdot,\xi)&:=\sum_{i=1}^mB_{n+2,m+1}(a_1,\ldots,a_{i},\tau_\xi a_i,\ldots, \tau_\xi a_{m})[\tau_\xi b_1,\ldots\tau_\xi b_n,a_i+\tau_\xi a_i,a_i-\tau_\xi a_i, \tau_\xi\oo].
\end{align*}
Hence,
\begin{align}\label{E:1}
[B_{n,m}[\oo]]_{W^{r}_{p'}}^{p'}\leq 3^{p'}\sum_{\ell=1}^3\int_{\R}\frac{\|T_\ell(\cdot,\xi)\|_{p'}^{p'}}{|\xi|^{1+rp'}}\, d\xi
\end{align} 
and  Lemma \ref{L:MP0} (with $p=p'$) yields
\begin{align}\label{AAE:1}
 \int_{\R}\frac{\|T_1(\cdot,\xi)\|_{p'}^{p'}}{|\xi|^{1+rp'}}\, d\xi&\leq C^p\Big(\prod_{i=1}^{n}\|b_i'\|_{\infty}^{p'} \Big) \int_{\R}\frac{\|  \oo-\tau_\xi\oo\|_{p'}^{p'}}{|\xi|^{1+rp'}}\, d\xi
 \leq \Big(C[\oo]_{W^{r}_{p'}}\prod_{i=1}^{n}\|b_i'\|_{\infty}\Big)^{p'}.
\end{align} 
Furthermore, using \eqref{REF1}, Lemma~\ref{L:AL0} (with $p=p'$, $t=s+1-2/p$, and $t'=s+1-r-2/p$), and the embedding $W^s_p(\R)\hookrightarrow W^{s+1-2/p}_{p'}(\R)$, we deduce that
\begin{equation}\label{AAE:2}
\begin{aligned}
\int_{\R}\frac{\|T_2(\cdot,\xi)\|_{p'}^{p'}}{|\xi|^{1+rp'}}\, d\xi 
&\leq C^p\|\oo\|_{W^{r}_{p'}}^{p'}\sum_{i=1}^n  \Big(\int_{\R}\frac{\|  b_i-\tau_\xi b_i\|_{W^{t'}_{p'}}^{p'}}{|\xi|^{1+(t-t')p'}}\, d\xi\Big)\prod_{j=1,j\neq i}^n\|b_j\|_{W^{s}_p}^{p'}  \\[1ex]
&\leq \Big(C\|\oo\|_{W^{r}_{p'}}\prod_{i=1}^{n}\|b_i\|_{W^{s}_p} \Big)^{p'}
\end{aligned}
\end{equation}
and, by similar arguments,
\begin{equation}\label{AAE:3}
\begin{aligned}
 \int_{\R}\frac{\|T_3(\cdot,\xi)\|_{p'}^{p'}}{|\xi|^{1+rp'}}\, d\xi\leq \Big(C\|\oo\|_{W^{r}_{p'}}\prod_{i=1}^{n}\|b_i\|_{W^{s}_p} \Big)^{p'}.
\end{aligned}
\end{equation}
The relations \eqref{E:1}-\eqref{AAE:3} lead to the desired estimate.
The local Lipschitz continuity  property follows from \eqref{spr3} and \eqref{REF1'2}.
\end{proof}

Together with Lemma~\ref{L:MP2} (with $r=2-s\in(0,1-1/p)$) we obtain the following result.
 \begin{lemma}\label{L:MP3}
Given $p\in(1,2]$,  $s\in(1 +1/p,2)$, $n,\, m\in\N,$ and   ${a_1,\ldots, a_m \in W^s_p(\R)}$, 
  there exists a constant~$C=C(n,\, m,\,s,\,p,\,\max_{1\leq i\leq m}\|a_i\|_{W^s_p})$ such that
\begin{align} 
\| B_{n,m}(a_1,\ldots, a_{m})[b_1,\ldots, b_n,\oo]\|_{W^{s-2}_{p}}\leq C \|\oo\|_{W^{s-2}_{p}}\prod_{i=1}^{n}\|b_i\|_{W^{s}_p} \label{REF1'''}
\end{align}
for all $ b_1,\ldots, b_n\in W^s_p(\R)$ and $\oo\in L_{p}(\R).$

Moreover,   $  B_{n,m}\in {\rm C}^{1-}((W^s_p(\R))^m,\kL^{n}_{\rm sym}(W_p^{s}(\R),\kL(  W^{s-2}_{p}(\R)))).$ 
 \end{lemma}
\begin{proof}
We recall  that $W^{s-2}_p(\R)=(W^{2-s}_{p'}(\R))'$. 
Let ${B_{n,m}:=B_{n,m}(a_1,\ldots, a_{m})[b_1,\ldots, b_n,\cdot]}$. 
It is not difficult to prove that the  $L_2$-adjoint of  $B_{n,m} $ is the operator $-B_{n,m}$.
Therefore, given $\oo,\,\varphi\in  {\rm C}^\infty_0(\R)$,  we obtain, in view of Lemma~\ref{L:MP2},  
\begin{align*}
|\langle B_{n,m}[\oo]|\varphi\rangle_{W^{s-2}_p(\R)\times W^{2-s}_{p'}(\R)}|&=|\langle \oo|  B_{n,m}[\varphi]\rangle_{W^{s-2}_p(\R)\times W^{2-s}_{p'}(\R)}|\\[1ex]
&\leq C \|\oo\|_{W^{s-2}_p}\|\varphi\|_{W^{2-s}_{p'}}\prod_{i=1}^{n}\|b_i\|_{W^{s}_p}.
\end{align*}
The estimate \eqref{REF1'''} follows via a standard density argument.
 The Lipschitz continuity property is a consequence of \eqref{REF1'''} and of  \eqref{spr3}.
\end{proof}

%%%%%%%%%%%%%%%%%%%%%%%%%%%%%%%%%%%%%%%%%%%%%%%%%%
%%%%%%%%%%%%%%%%%%%%%%%%%%%%%%%%%%%%%%%%%%%%%%%%%%
%%%%%%%%%%%%%%%%%%%%%%%%%%%%%%%%%%%%%%%%%%%%%%%%%%
%%%%%%%%%%%%%%%%%%%%%%%%%%%%%%%%%%%%%%%%%%%%%%%%%%
\section{A functional analytic framework for the Muskat problem}\label{Sec:2}
%%%%%%%%%%%%%%%%%%%%%%%%%%%%%%%%%%%%%%%%%%%%%%%%%%
%%%%%%%%%%%%%%%%%%%%%%%%%%%%%%%%%%%%%%%%%%%%%%%%%%
%%%%%%%%%%%%%%%%%%%%%%%%%%%%%%%%%%%%%%%%%%%%%%%%%%
%%%%%%%%%%%%%%%%%%%%%%%%%%%%%%%%%%%%%%%%%%%%%%%%%%
In this section we take advantage of the mapping properties established in Section~\ref{Sec:1} and formulate the Muskat problem \eqref{P} as a quasilinear evolution problem
in a suitable functional analytic setting, see \eqref{AF}-\eqref{goal1}.
Afterwards, we show that the problem is of parabolic type. 
This enables us to employ  theory for such evolution equations as presented in~\cite{Am93, MW20}  to establish our main result in Theorem~\ref{MT1}.
The quasilinear structure of~\eqref{P} is  due to the quasilinearity of the curvature operator, the latter being established in Lemma~\ref{L:1}.

\begin{lemma}\label{L:1}
Given $p\in(1,\infty)$ and $s\in(1+1/p,2)$,   the operator~${\kappa(\cdot)[\cdot]}$ defined in~\eqref{qlk} satisfies
 $\kappa\in {\rm C}^\infty(W^s_p(\R),\kL(W^{s+1}_p(\R), W^{s-1}_p(\R))).$
\end{lemma}
\begin{proof}
The arguments are similar to those presented in \cite[Appendix~C]{MP20x} and are therefore omitted. 
\end{proof}

The Muskat problem \eqref{P} can thus be formulated as the evolution problem
\begin{align}\label{AF}
\frac{df}{dt}(t)=\Phi(f(t))[f(t)],\quad t>0,\qquad f(0)=f_0,
\end{align}
where
\begin{align}\label{PHI}
\Phi(f)[h]:=\frac{k}{2\mu}\bB(f)[\sigma(\kappa(f)[h])'-\Delta_\rho h'],
\end{align}
with $\bB$   introduced in \eqref{OpB}.
Arguing as in \cite[Appendix~C]{MP20x}, we may infer from Lemma~\ref{L:MP3} that, given $n,\, m\in\N$, $p\in(1,2]$, and $s\in(1+1/p,2),$  we have
\[
[f\mapsto B^0_{n,m}(f)]\in {\rm C}^\infty(W^s_p(\R), \kL(W^{s-2}_p(\R))).
\]
This property and   Lemma~\ref{L:1} combined yield
\begin{align}\label{goal1}
 \Phi\in {\rm C}^\infty(W^s_p(\R),\kL(W^{s+1}_p(\R), W^{s-2}_p(\R)))
\end{align}
for all $p\in(1,2]$ and $s\in(1+1/p,2).$

Let  $p\in(1,2]$, $s\in (1+1/p,2)$, and $f\in W^s_p(\R)$ be fixed in the remaining of this section.
The analysis below is devoted to showing that the linear operator~$\Phi(f)$, viewed as an unbounded operator in~$W^{s-2}_p(\R)$ and with definition domain~$W^{s+1}_p(\R)$,
 is the generator of an analytic semigroup in~$\kL(W^{2-s}_p(\R))$,  which  writes in the notation  used in \cite{Am95} as
\begin{align}\label{goal2}
-\Phi(f)\in\kH(W^{s+1}_p(\R), W^{s-2}_p(\R)).
\end{align}
This property is established in  Theorem~\ref{T:GP} below and it identifies the quasilinear evolution problem \eqref{AF} as being of parabolic type.
To start, we note that $\pi^{-1}B_{0,0 }= H$, where $H$ is the Hilbert transform,  and therefore
\begin{align}\label{phi0}
\Phi(0)=\frac{k\sigma}{2 \mu}H\circ\frac{d^3}{dx^3}-\frac{k\Delta_\rho}{2 \mu}H\circ\frac{d}{dx}=-\frac{k\sigma}{2 \mu}\Big(\frac{d^4}{dx^4}\Big)^{3/4}-\frac{k\Delta_\rho}{2 \mu}\Big(-\frac{d^2}{dx^2}\Big)^{1/2},
\end{align}
where $(d^4/dx^4)^{3/4}$ denotes the Fourier multiplier with symbol $m(\xi):=|\xi|^3$ and $(-d^2/dx^2)^{1/2}$ is the Fourier multiplier with symbol $m(\xi):=|\xi|$.
We shall locally approximate the operator~$\Phi(\tau f)$, with $\tau\in[0,1]$,   by  certain Fourier multipliers~$\bA_{j,\tau}$. 
Therefore we choose   for each~${\e\in(0,1)}$ a so-called finite~$\e$-localization family, that is a set 
\[\{(\pi_j^\e,x_j^\e)\,:\, -N+1\leq j\leq N\}\subset {\rm C}^\infty(\mathbb{R},[0,1])\times\R\]
 such that
\begin{align*}
\bullet\,\,\,\, \,\,& \text{$\sum_{j=-N+1}^N(\pi_j^\e)^2=1$ and $\|\pi_j^\e)^{(k)}\|_\infty\leq C\e^{-k}$ for all $ k\in\N,\, -N+1\leq j\leq N$;} \\[1ex]
\bullet\,\,\,\, \,\,  & \text{$ \supp \pi_j^\e $ is an interval of length $\e$ for $|j|\leq N-1$,} \, \, \text{$\supp \pi_{N}^\e\subset\{|x|>1/\e\}$;} \\[1ex]
\bullet\,\,\,\, \,\, &\text{ $ \pi_j^\e\cdot  \pi_l^\e=0$ if $[|j-l|\geq2,\, \max\{|j|,\, |l|\}\leq N-1]$ or $[|l|\leq N-2,\, j=N];$} \\[1ex]
 \bullet\,\,\,\, \,\, &x^\e_j\in\supp\pi_j^\e,\; |j|\leq N-1. 
 \end{align*} 
 The real number $x_N^\e$ plays no role in the analysis below.
To each  finite $\e$-localization family we associate  a second family  $\{\chi_j^\e\,:\, -N+1\leq j\leq N\}\subset {\rm C}^\infty(\mathbb{R},[0,1])$ such that
\begin{align*}
\bullet\,\,\,\, \,\,  &\text{$\chi_j^\e=1$ on $\supp \pi_j^\e$, $-N+1\leq j\leq N$, and $\supp\chi_N^\e\subset\{|x|>1/\e-\e\}$}; \\[1ex]
\bullet\,\,\,\, \,\,  &\text{$\supp \chi_j^\e$ is an interval  of length $3\e$ and with the same midpoint as $ \supp \pi_j^\e$, $|j|\leq N-1$.} 
\end{align*} 

To each finite $\e$-localization family we associate  a norm  on $W^r_p(\R),$ $r\in\R$, which is equivalent to the standard  norm.

\begin{lemma}\label{L:EN}
Let $\e\in(0,1)$ and let  $\{(\pi_j^\e,x_j^\e)\,:\, -N+1\leq j\leq N\}$ be a finite $\e$-localization family.  Given $p\in(1,\infty)$ and $r\in\R$, there exists $c=c(\e,r,p)\in(0,1)$ such that
\[
c\|f\|_{W^r_p}\leq \sum_{j=-N+1}^N\|\pi_j^\e f\|_{W^r_p}\leq c^{-1}\|f\|_{W^r_p},\qquad f\in W^r_p(\R). 
\]
\end{lemma} 
\begin{proof}
The claim follows from the  fact that  $\pi_j^\e\in {\rm C}^\infty(\R)$ is a pointwise multiplier for~$W^r_p(\R)$.
\end{proof}

The next result is the main step in the proof of \eqref{goal2}.

\begin{thm}\label{T:AP} 
Let $p\in(1,2]$, $s\in(1+1/p,2)$, $\rho\in(0,\min\{(s-1-1/p)/2),2-s\})$, and~$\nu>0$ be given. 
Then, there exist $\e\in(0,1)$, a $\e$-locali\-za\-tion family  $\{(\pi_j^\e,x_j^\e)\,:\, -N+1\leq j\leq N\} $,  a constant $K=K(\e)$, 
and   bounded operators 
$$
\bA_{j,\tau}\in\kL(W^{s+1}_p(\R), W^{s-2}_p(\R)), \qquad\text{$j\in\{-N+1,\ldots,N\}$ and $\tau\in[0,1]$,} 
$$
 such that 
 \begin{equation}\label{D1}
  \|\pi_j^\e \Phi(\tau f) [h]-\bA_{j,\tau}[\pi^\e_j h]\|_{W^{s-2}_p}\leq \nu \|\pi_j^\e h\|_{W^{s+1}_p}+K\|  h\|_{W^{ s+1-\rho}_p}
 \end{equation}
 for all $ -N+1\leq j\leq N$, $\tau\in[0,1],$  and  $h\in W^{s+1}_p(\R)$. 
 The operators $\bA_{j,\tau}$ are defined~by 
  \begin{align*} 
 \bA_{j,\tau }:=- \alpha_\tau(x_j^\e) \Big(\frac{d^4}{dx^4}\Big)^{3/4}, \quad |j|\leq N-1, \qquad\bA_{N,\tau }:= -  \frac{k\sigma}{2 \mu} \Big(\frac{d^4}{dx^4}\Big)^{3/4},
 \end{align*}
and   $\alpha_\tau:=(k\sigma/(2\mu))(1+\tau^2f'^2)^{-3/2}$.
\end{thm}
\begin{proof} In this proof  we denote by $C$  constants that do not depend on $\e$ and we write $K$ for constants that depend on $\e$.

Given  $-N+1\leq j\leq N$, $\tau\in[0,1],$  and  $h\in W^{s+1}_p(\R)$,   Lemma~\ref{L:MP0}  yields 
\begin{align}\label{uest0}
\|\pi_j^\e\bB(\tau f)[h']\|_{W^{s-2}_p}\leq \|\pi_j^\e\bB(\tau f)[h']\|_{p}\leq C\|h'\|_{p}\leq C\|h\|_{W^{s+1-\rho}_p},
\end{align}
while, using some elementary  arguments, we get
\begin{equation}\label{uest}
\begin{aligned}
\|\pi_j^\e(\kappa(f)[h])'\|_{W^{s-2}_p}&\leq  \|\pi_j^\e \kappa(f)[h]\|_{W^{s-1}_p}+K\| \kappa(f)[h]\|_{p}\\[1ex]
&\leq C_0 \|\pi_j^\e h\|_{W^{s+1}_p}+K\|h\|_{W^{s+1-\rho}_p}.
\end{aligned}
\end{equation}
Below we take advantage of \eqref{uest}  when considering the leading order term $\bB(\tau f)[(\kappa(\tau f)[h])']$ of $~\Phi(\tau f) [h]$.
Let $C_1:=C_0k\sigma/2\pi\mu.$\medskip

\noindent{\em Step 1: The case $|j|\leq N-1$.} 
Given $|j|\leq N-1$, we infer from Lemma~\ref{L:AL3} and \eqref{uest}  that
\begin{equation}\label{ST1a}
\begin{aligned}
&\hspace{-0,45cm}\Big\|\pi_j^\e B_{0,1}^0(\tau f) [(\kappa(\tau f)[h])']-\frac{1}{1+\tau^2f'^2(x_j^\e)}B_{0,0}[\pi^\e_j (\kappa(\tau f)[h])']\Big\|_{W^{s-2}_p}\\[1ex]
&\leq  \frac{\nu}{4C_1} \|\pi_j^\e(\kappa(\tau f)[h])'\|_{W^{s-2}_p}+K\|(\kappa(\tau f)[h])'\|_{W^{  s-2-\rho}_p}\\[1ex]
&\leq  \frac{\nu C_0}{4C_1} \|\pi_j^\e h\|_{W^{s+1}_p}+K\|h\|_{W^{ s+1-\rho}_p}
\end{aligned} 
\end{equation}
provided that $\e$ is sufficiently small.
Besides, we have 
\begin{equation*}
\Big\|\pi_j^\e \tau f'B_{1,1}^0(\tau f) [(\kappa(\tau f)[h])']-\frac{\tau^2f'^2(x_j^\e)}{1+\tau^2f'^2(x_j^\e)}B_{0,0}[\pi^\e_j (\kappa(\tau f)[h])']\Big\|_{W^{s-2}_p}\leq T_1+T_2+ T_3,
\end{equation*}
where 
\begin{equation*}
\begin{aligned}
T_1&:=\|\chi_j^\e (f'-f'(x_j^\e))  B_{1,1}^0(\tau f)[\pi_j^\e(\kappa(\tau f)[h])']\|_{W^{s-2}_p},\\[1ex]
T_2&:=\|\chi_j^\e (f'-f'(x_j^\e)) (\pi_j^\e B_{1,1}^0(\tau f)[(\kappa(\tau f)[h])']-B_{1,1}^0(\tau f)[\pi_j^\e (\kappa(\tau f)[h])'])\|_{W^{s-2}_p},\\[1ex]
T_3&:=\|f'\|_\infty\Big\|\pi_j^\e B_{1,1}^0(\tau f) [(\kappa(\tau f)[h])']-\frac{\tau f'(x_j^\e)}{1+\tau^2f'^2(x_j^\e)}B_{0,0}[\pi^\e_j (\kappa(\tau f)[h])']\Big\|_{W^{s-2}_p}.
\end{aligned} 
\end{equation*}
For $\e$ sufficiently small to guarantee that 
$$\|\chi_j^\e (f'-f'(x_j^\e))\|_\infty<\frac{\nu }{40C_1}\Big(\max_{\tau\in[0,1]}\|B_{1,1}^0(\tau f)\|_{\kL(W^{s-2}_p(\R))}\Big)^{-1},$$
  it follows from \eqref{EQ:Mul2} (with $r=s-1$),  Lemma~\ref{L:MP3}, and \eqref{uest} (if $\chi_j^\e (f'-f'(x_j^\e))$ is not identically zero, otherwise the estimate is trivial) that
  \begin{align}\label{ST1bi}
  T_1\leq \frac{\nu C_0}{8C_1} \|\pi_j^\e h\|_{W^{s+1}_p}+K\|h\|_{W^{ s+1-\rho}_p}.
  \end{align}
As $\chi_j^\e (f'-f'(x_j^\e))\in W^{s-1}_p(\R)$ is a pointwise multiplier for $W^{s-2}_p(\R),$ cf. \eqref{mult},  Lemma~\ref{L:AL1} yields
\begin{equation}\label{ST1bii}
T_2\leq K\|(\kappa(\tau f)[h])'\|_{W^{ s-2-\rho}_p}\leq K\|h\|_{W^{ s+1-\rho}_p}.
\end{equation}
Finally, if $\e$ is sufficiently small, we may argue as in the derivation of \eqref{ST1a} to get
\begin{equation}\label{ST1biii}
T_3\leq  \frac{\nu C_0}{8C_1}\|\pi_j^\e h\|_{W^{s+1}_p}+K\|h\|_{W^{ s+1-\rho}_p}.
\end{equation} 
Gathering \eqref{ST1bi}-\eqref{ST1biii}, we conclude that
\begin{equation}\label{ST1b}
\begin{aligned}
&\hspace{-0,45cm}\Big\|\pi_j^\e \tau f'B_{1,1}^0(\tau f) [(\kappa(\tau f)[h])']-\frac{\tau^2f'^2(x_j^\e)}{1+\tau^2f'^2(x_j^\e)}B_{0,0}[\pi^\e_j (\kappa(\tau f)[h])']\Big\|_{W^{s-2}_p}\\[1ex]
&\leq  \frac{\nu C_0}{4C_1} \|\pi_j^\e h\|_{W^{s+1}_p}+K\|h\|_{W^{ s+1-\rho}_p}.
\end{aligned} 
\end{equation} 
We now combine \eqref{uest0}, \eqref{ST1a}, and \eqref{ST1b} and  obtain that 
\begin{align}\label{aA}
\Big\|\pi_j^\e \Phi(\tau f) [h]-\frac{k\sigma}{2\pi\mu}B_{0,0}[\pi^\e_j (\kappa(\tau f)[h])']\Big\|_{W^{s-2}_p}\leq \frac{\nu}{2} \|\pi_j^\e h\|_{W^{s+1}_p}+K\|  h\|_{W^{ s+1-\rho}_p}.
\end{align}

As a final step  we show that, if $\e$ sufficiently small, then
\begin{align}\label{aB}
\Big\|B_{0,0}[\pi^\e_j (\kappa(\tau f)[h])']-\frac{B_{0,0}[(\pi^\e_j h)''']}{(1+\tau^2f'^2(x_j^\e))^{3/2}}\Big\|_{W^{s-2}_p}\leq \frac{\nu C_0}{2C_1} \|\pi_j^\e h\|_{W^{s+1}_p}+K\|  h\|_{W^{ s+1-\rho}_p}
\end{align}
for all $h\in W^{s+1}_p(\R)$, $\tau\in[0,1]$, and $|j|\leq N-1.$
To start, we note that 
\[
B_{0,0}[\pi^\e_j (\kappa(\tau f)[h])']=B_{0,0}[(\pi^\e_j \kappa(\tau f)[h])']-B_{0,0}[(\pi^\e_j)' \kappa(\tau f)[h]],
\]
and Lemma \ref{L:MP3} yields for $-N+1\leq j\leq N$ that
\begin{align}\label{b1}
\|B_{0,0}[(\pi^\e_j)' \kappa(\tau f)[h]]\|_{W^{s-2}_p}\leq K\|  h\|_{W^{ s+1-\rho}_p}.
\end{align}
Moreover, letting $C_2:=\|B_{0,0}\|_{\kL(W^{s-2}_p(\R))}$,  the algebra property of $W^{s-1}_p(\R)$ leads us to
\begin{equation}\label{b2''}
\begin{aligned}
&\hspace{-0.5cm}\Big\|B_{0,0}[(\pi^\e_j \kappa(\tau f)[h])']-\frac{B_{0,0}[(\pi_j^\e h)''']}{(1+\tau^2 f'^2(x_j^\e))^{3/2}}\Big\|_{W^{s-2}_p}\\[1ex]
&\leq C_2\Big\|  \frac{ \pi^\e_j h''}{(1+\tau^2f'^2)^{3/2}}  -\frac{ (\pi_j^\e h) ''}{(1+\tau^2 f'^2(x_j^\e))^{3/2}}\Big\|_{W^{s-1}_p}\\[1ex]
&\leq C_2\Big\|  \Big(\frac{1}{(1+\tau^2f'^2)^{3/2}}  -\frac{1}{(1+\tau^2 f'^2(x_j^\e))^{3/2}}\Big)(\pi_j^\e h) ''\Big\|_{W^{s-1}_p}+K \|h\|_{W^s_p}.
\end{aligned}
\end{equation}
Using \eqref{MES2} together with  the identity $\chi_j^\e\pi_j^\e=\pi_j^\e$, for $\e$   sufficiently small we get
\begin{equation}\label{b2}
\begin{aligned}
&\hspace{-0.5cm}\Big\|  \Big(\frac{1}{(1+\tau^2f'^2)^{3/2}}  -\frac{1}{(1+\tau^2 f'^2(x_j^\e))^{3/2}}\Big)(\pi_j^\e h) ''\Big\|_{W^{s-1}_p}\\[1ex]
&\leq 2\Big\|  \chi_j^\e\Big(\frac{1}{(1+\tau^2f'^2)^{3/2}}  -\frac{1}{(1+\tau^2 f'^2(x_j^\e))^{3/2}}\Big)\Big\|_\infty\|\pi_j^\e h\|_{W^{s+1}_p}+K\| h\|_{W^{s+1-\rho}_p}\\[1ex]
&\leq \frac{\nu C_0}{2C_1C_2} \|\pi_j^\e h\|_{W^{s+1}_p}+K\|h\|_{W^{ s+1-\rho}_p}.
\end{aligned}
\end{equation}
Gathering \eqref{b1}-\eqref{b2}, we conclude that \eqref{aB} holds true. 
 The desired estimate~\eqref{D1} follows now, for $|j|\leq N-1,$ by combining  \eqref{aA} and \eqref{aB}.\medskip

\noindent{\em Step 2: The case $j=N$.}  Similarly to \eqref{ST1a}, we obtain from Lemma~\ref{L:AL4} and \eqref{uest} that
\begin{equation}\label{ST1aN}
\begin{aligned}
&\hspace{-0.5cm}\|\pi_N^\e B_{0,1}^0(\tau f) [(\kappa(\tau f)[h])']-B_{0,0}[\pi^\e_N (\kappa(\tau f)[h])']\|_{W^{s-2}_p}\\[1ex]
&\leq  \frac{\nu C_0}{4C_1} \|\pi_N^\e h\|_{W^{s+1}_p}+K\|h\|_{W^{ s+1-\rho}_p}
\end{aligned}
\end{equation}
for all $h\in W^{s+1}_p(\R)$ and $\tau\in[0,1]$,  provided that $\e$ is sufficiently small. 
Moreover,
\begin{equation*}
\|\pi_N^\e \tau f'B_{1,1}^0(\tau f) [(\kappa(\tau f)[h])']\|_{W^{s-2}_p}\leq T_a+T_b,
\end{equation*}
where
\begin{equation*}
\begin{aligned}
T_a&:=\|\chi_N^\e f' B_{1,1}^0(\tau f)[\pi_N^\e(\kappa(\tau f)[h])']\|_{W^{s-2}_p},\\[1ex]
T_b&:=\|\chi_N^\e f' \big(\pi_N^\e B_{1,1}^0(\tau f)[(\kappa(\tau f)[h])']-B_{1,1}^0(\tau f)[\pi_N^\e (\kappa(\tau f)[h])']\big)\|_{W^{s-2}_p}.
\end{aligned} 
\end{equation*}
Since $\chi_N^\e f'\in W^{s-1}_p(\R)$ is a pointwise multiplier for $W^{s-2}_p(\R),$  cf. \eqref{mult}, Lemma \ref{L:AL1} yields
\begin{equation}\label{ST1biiN}
T_b\leq K\|(\kappa(\tau f)[h])'\|_{W^{ s-2-\rho}_p}\leq K\|h\|_{W^{ s+1-\rho}_p}.
\end{equation}
Because $f'\in W^{s-1}_p(\R)$ vanishes at infinity, for $\e$ sufficiently small to ensure that
\[
\|\chi_N^\e f'\|_\infty<\frac{\nu }{20C_1}\Big(\max_{\tau\in[0,1]}\|B_{1,1}^0(\tau f)\|_{\kL(W^{s-2}_p(\R))}\Big)^{-1},
\]
  it follows from \eqref{EQ:Mul2} (with $r=s-1$),  Lemma~\ref{L:MP3}, and \eqref{uest} (if $\chi_N^\e f'$ is not identically zero, otherwise the estimate is trivial) that
  \begin{align}\label{ST1biN}
  T_a\leq \frac{\nu C_0}{4C_1} \|\pi_N^\e h\|_{W^{s+1}_p}+K\|h\|_{W^{ s+1-\rho}_p}.
  \end{align}
  Gathering \eqref{uest0} and \eqref{ST1aN}-\eqref{ST1biN}, we have shown that if $\e$ is sufficiently small, then
  \begin{align}\label{aAN}
\Big\|\pi_N^\e \Phi(\tau f) [h]-\frac{k\sigma}{2\pi\mu}B_{0,0}[\pi^\e_N (\kappa(\tau f)[h])']\Big\|_{W^{s-2}_p}\leq \frac{\nu}{2} \|\pi_N^\e h\|_{W^{s+1}_p}+K\|  h\|_{W^{ s+1-\rho}_p}
\end{align}
for all $h\in W^{s+1}_p(\R)$ and $\tau\in[0,1]$.
It remains to show that for $\e$ sufficiently small  
\begin{align}\label{aBN}
\|B_{0,0}[\pi^\e_N (\kappa(\tau f)[h])']-B_{0,0}[(\pi^\e_N h)''']\|_{W^{s-2}_p}\leq \frac{\nu C_0}{2C_1} \|\pi_N^\e h\|_{W^{s+1}_p}+K\|  h\|_{W^{ s+1-\rho}_p}
\end{align}
for all $h\in W^{s+1}_p(\R)$ and $\tau\in[0,1]$.
Arguing as in the first step (see \eqref{b2''}), we find in view of   \eqref{b1}   that
\begin{align*}
  &\hspace{-0.5cm} \|B_{0,0}[\pi^\e_N( \kappa(\tau f)[h])']- B_{0,0}[(\pi_N^\e h)''']\|_{W^{s-2}_p} \\[1ex]
   &\leq C_2\Big\|  \Big(\frac{1}{((1+\tau^2f'^2)^{3/2}}  - 1\Big)(\pi_N^\e h)''\Big\|_{W^{s-1}_p}+K\|  h\|_{W^{ s+1-\rho}_p}.
\end{align*}
Using \eqref{MES2} and the fact that $f'$ vanishes at infinity, for $\e$ sufficiently small, we obtain
\begin{equation*} 
\begin{aligned}
&\hspace{-0.5cm}C_2\Big\|  \Big(\frac{1}{((1+\tau^2f'^2)^{3/2}}  - 1\Big)(\pi_N^\e h)''\Big\|_{W^{s-1}_p}\\[1ex]
&\leq 2C_2 \|  \chi_N^\e (1  -(1+\tau^2f'^2)^{3/2} ) \|_\infty\|\pi_N^\e h\|_{W^{s+1}_p}+K\| h\|_{W^{s+1-\rho}_p}\\[1ex]
&\leq \frac{\nu C_0}{2C_1} \|\pi_N^\e h\|_{W^{s+1}_p}+K\|h\|_{W^{ s+1-\rho}_p}.
\end{aligned}
\end{equation*}
This proves \eqref{aBN}. 
The claim \eqref{D1} for $j= N$ follows now directly from   \eqref{aAN} and \eqref{aBN}.
\end{proof}

 We now consider a  class of Fourier multipliers related to the multipliers from Theorem~\ref{T:AP}.

\begin{lemma}\label{L:GAP} 
Let $\eta\in(0,1)$. Given  $\alpha\in[\eta,1/\eta]$, let
\begin{align*} 
 \bA_{\alpha}:=- \alpha\Big(\frac{d^4}{dx^4}\Big)^{3/4}.
 \end{align*}
Then, there exits a constant $\kappa_0=\kappa_0(\eta)\geq 1$ such that
 \begin{align}
\bullet &\quad \lambda-\bA_{\alpha}\in {\rm Isom}(W^{s+1}_p(\R),W^{s-2}_p(\R))\quad  \forall\, \re\lambda\geq 1,\label{L:FM1}\\[1ex]
\bullet &\quad  \kappa_0\|(\lambda-\bA_{\alpha})[h]\|_{W^{s-2}_p}\geq |\lambda|\cdot\|h\|_{W^{s-2}_p}+\|h\|_{W^{s+1}_p} \quad \forall\, h\in W^{s+1}_p(\R),\, \re\lambda\geq 1\label{L:FM2}.
\end{align}
\end{lemma}
\begin{proof}
We first consider the  realizations
$$\bA_{\alpha}\in \kL(W^2_p(\R),W^{-1}_p(\R))\quad\text{and}\quad\bA_{\alpha}\in\kL(W^3_p(\R),L_p(\R)).$$ 
Since $W^k_p(\R)=H^k_p(\R)$, $k\in\Z$, Mikhlin's multiplier theorem, cf. e.g. \cite[Theorem 4.23]{AB12}, shows that the properties \eqref{L:FM1}-\eqref{L:FM2} (in the appropriate spaces) 
are valid for these realizations.
 Then,  using the interpolation property \eqref{IP}, we obtain that \eqref{L:FM1}-\eqref{L:FM2} hold true.
\end{proof}

The next result provides the generator property announced in \eqref{goal2}.
\begin{thm}\label{T:GP}
Given $p\in (1,2]$, $s\in(1+1/p,2)$, and $f\in W^{s}_p(\R)$ we have
\begin{align*}
-\p\Phi(f)\in\mathcal{H}(W^{s+1}_p(\R), W^{s-2}_p(\R)).
\end{align*}
\end{thm}
\begin{proof}  
Fix $\rho\in(0,\min\{(s-1-1/p)/2),2-s\})$. 
Since $f'\in W^{s-1}_p(\R)$ is bounded, there exists a constant $\eta>0$ with the property that the function $\alpha_\tau$, $\tau\in[0,1]$, from Theorem~\ref{T:AP} satisfies
$\eta\leq |\alpha_\tau|\leq  1/\eta$ for all $\tau\in[0,1]$. 
Hence, regardless of  $\e>0,$ the operators~$A_{j,\tau},$ with~$-N+1\leq j\leq N$ and $\tau\in[0,1],$ defined in   Theorem~\ref{T:AP}  satisfy \eqref{L:FM1}-\eqref{L:FM2} 
with a constant $\kappa_0\geq1.$ 

We set $\nu:=(2\kappa_0)^{-1}$.
 Theorem \ref{T:AP} provides  an $\e\in(0,1) $,  a $\e$-localization family, a  constant~$K=K(\e)>0$,
and bounded operators~$\bA_{j,\tau}\in\kL(W^{s+1}_p(\R), W^{s-2}_p(\R))$,~${-N+1\leq j\leq N}$ and~$\tau\in[0,1],$  such that
 \begin{equation*} 
  2\kappa_0\|\pi_j^\e\Phi(\tau f)[h]-\bA_{j,\tau}[\pi^\e_j h]\|_{W^{s-2}_p}\leq \|\pi_j^\e h\|_{W^{s+1}_p}+2\kappa_0 K\|  h\|_{W^{s+1-\rho}_p}
 \end{equation*}
for all $-N+1\leq j\leq N$, $\tau\in[0,1],$  and  $h\in W^{s+1}_p(\R)$.
Besides,  Lemma \ref{L:GAP}  yields
  \begin{equation*} 
    2\kappa_0\|(\lambda-\bA_{j,\tau})[\pi^\e_jh]\|_{W^{s-2}_p}\geq 2|\lambda|\cdot\|\pi^\e_jh\|_{W^{s-2}_p}+ 2\|\pi^\e_j h\|_{W^{s+1}_p}
 \end{equation*}
 for all $-N+1\leq j\leq N$, $\tau\in[0,1],$  $\re \lambda\geq 1$, and $h\in W^{s+1}_p(\R)$.
The latter  inequalities imply
 \begin{align*}
   2\kappa_0\|\pi_j^\e(\lambda-\Phi(\tau f))[h]\|_{W^{s-2}_p}\geq& 2\kappa_0\|(\lambda-\bA_{j,\tau})[\pi^\e_j h]\|_{W^{s-2}_p}-2\kappa_0\|\pi_j^\e\Phi(\tau f)[h]-\bA_{j,\tau}[\pi^\e_j h]\|_{W^{s-2}_p}\\[1ex]
   \geq& 2|\lambda|\cdot\|\pi^\e_j h\|_{W^{s-2}_p}+ \|\pi^\e_j h\|_{W^{s+1}_p}-2\kappa_0K\|  f\|_{W^{s-1-\rho}_p}.
 \end{align*}
 Summing  up over $j$  and using Lemma~\ref{L:EN}, the interpolation property \eqref{IP},  and Young's inequality  we find constants  $\kappa=\kappa(f)\geq1$  and $\omega=\omega( f)>0 $ such that 
  \begin{align}\label{KDED}
   \kappa\|(\lambda-\Phi(\tau f ))[h]\|_{W^{s-2}_p}\geq |\lambda|\cdot\|h\|_{W^{s-2}_p}+ \| h\|_{W^{s+1}_p}
 \end{align}
for all   $\tau\in[0,1],$   $\re \lambda\geq \omega$, and  $h\in W^{s+1}_p(\R)$.

Recalling \eqref{phi0} and arguing as in Lemma \ref{L:GAP} we may choose $\omega$ sufficiently large to guarantee that $\omega-\Phi(0) \in {\rm Isom}(W^{s+1}_p(\R), W^{s-2}_p(\R))$.
The method of continuity together with~\eqref{KDED} consequently yields 
\begin{align}\label{DEDK2}
   \omega-\Phi(f)\in {\rm Isom}(W^{s+1}_p(\R), W^{s-2}_p(\R)).
 \end{align}
The relations \eqref{KDED} (with $\tau=1$) and \eqref{DEDK2}  imply the desired claim, cf. \cite[Chapter I]{Am95}.
\end{proof}

We next present the proof of the main result. The arguments rely to a large extent on the theory of quasilinear parabolic problems presented in \cite{Am93, Am86, Am88} (see also \cite{MW20}).
Besides, in order to obtain the parabolic smoothing properties, we additionally employ a parameter 
trick which was successfully applied also to other problems, cf., e.g., \cite{An90, ES96, PSS15, AC11}.

\begin{proof}[Proof of Theorem~\ref{MT1}]
Fix 
\[
1+\frac{1}{p}<\ov s<s<2 \qquad \text{and}\qquad 0< \beta:=\frac{2}{3}<\alpha:=\frac{s-\ov s+2}{3}<1. 
\]
We further set $E_1:=W^{\ov s+1}_p(\R)$, $E_0:=W^{\ov s-2}_p(\R)$, and $E_\theta:=(E_0,E_1)_{\theta,p} $, with $\theta\in\{\alpha,\, \beta\}$.
Recalling the interpolation property \eqref{IP}, we have $E_\alpha=W^{s}_p(\R)$ and $E_\beta=W^{\ov s}_p(\R)$.
In view of \eqref{goal1} and \eqref{goal2} (with~${s=\ov s}$), we  then get
 \[
-\Phi \in {\rm C}^{\infty}(E_\beta,\kH(E_1, E_0)),
\]
 and the assumptions of \cite[Theorem 1.1]{MW20} are satisfied in the context of the Muskat problem~\eqref{AF}.
Applying  \cite[Theorem 1.1]{MW20}, we may conclude that \eqref{AF}  has for each $f_0\in W^s_p(\R)$ a  unique maximal classical solution  $f= f(\,\cdot\, ; f_0)$
  such that 
 \begin{equation} \label{reg123}
 f\in {\rm C}([0,T^+),W^s_p(\mathbb{R}))\cap {\rm C}((0,T^+), W^{{\ov s}+1}_p(\mathbb{R}))\cap {\rm C}^1((0,T^+), W^{{\ov s}-2}_p(\mathbb{R})).
  \end{equation}
  Moreover, if the solution belongs to the set
  \[
  \bigcup_{\eta\in(0,1)}{\rm C}^{\eta}([0,T^+),W^{\ov s}_p(\mathbb{R})),
  \]
then it is also unique, cf. \cite[Remark 1.2~(ii)]{MW20}.
We now prove that  each solution to \eqref{AF} that satisfies \eqref{reg123} belongs to ${\rm C}^{\eta}([0,T^+),W^{\ov s}_p(\mathbb{R}))$ with $\eta=(s-\ov s)/(3+s-\ov s)$.
Indeed, since $f\in {\rm C}([0,T^+),W^s_p(\mathbb{R}))$, we infer from  Lemma~\ref{L:MBC}  that
\[
\sup_{t\in(0,T]}\Big\|\frac{df}{dt}(t)\Big\|_{W^{\ov s-3}_p(\R)}<\infty
\]
for each $T\in(0,T^+)$. 
Since for $\theta:=3/(s-\ov s+3)$ we have $(W^{\ov s-3}_{p}(\R), W^{s}_{p}(\R))_{\theta,p}=W^{\ov s}_{p}(\R),$ cf. \eqref{IP},  \eqref{reg123} and the latter estimate now yield
\[
\|f(t_1)-f(t_2)\|_{W^{\ov s}_p}\leq C\|f(t_1)-f(t_2)\|_{W^{\ov s-3}_p}^{1-\theta} \leq C |t_1-t_2|^{(s-\ov s)/(3+s-\ov s)} \quad\text{for $0\leq t_1\leq t_2\leq T$.}
\]
Therewith the existence and uniqueness claim is   proven. 
Moreover, the  assertion (i) follows from the abstract theory (see the proof of \cite[Theorem~1.1]{MW20}).

The parabolic smoothing property established   at (ii) can be shown by arguing as in the more restrictive case considered in \cite[Theorem 1.3]{MBV19}.

Finally, in order to prove (iii), we assume  that $f= f(\,\cdot\, ; f_0):[0,T^+)\to W^{s}_p(\R)$ is a maximal classical solution  with $T^+<\infty$ and that
\[
\sup_{t\in[0,T^+)}\|f(t)\|_{W^s_p(\R)}<\infty.
\]
Using again Lemma~\ref{L:MBC} and arguing as above, we conclude that $f:[0,T^+)\to W^{\ov s}_p(\mathbb{R})$ is uniformly continuous.
Let now 
\[
1+\frac{1}{p}<\wt s<\ov s<2 \qquad \text{and}\qquad 0< \beta:=\frac{2}{3}<\alpha':=\frac{\ov s-\wt s+2}{3}<1. 
\]
Choosing $F_1:=W^{\wt s+1}_p(\R)$, $F_0:=W^{\wt s-2}_p(\R)$, and setting $F_\theta:=(F_0,F_1)_{\theta,p}$,  $\theta\in\{\alpha',\, \beta\}$,
 we have that~$F_\alpha=W^{\ov s}_p(\R)$ and $F_\beta=W^{\wt s}_p(\R)$.
Moreover  \eqref{goal1} and \eqref{goal2} (with~${s=\wt s}$), yield
 \[
-\Phi \in {\rm C}^{\infty}(F_\beta,\kH(F_1, F_0)).
\]
Thus, we may apply again \cite[Theorem 1.1]{MW20}~(iv)~($\alpha$) to~\eqref{AF} and conclude that $f$ can be extended to an interval $ [0,{\wt T}^+)$ with ${\wt T}^+>T^+$ and such that 
\begin{equation*} 
 f\in {\rm C}([0,\wt T^+),W^{\ov s}_p(\mathbb{R}))\cap {\rm C}((0,\wt T^+), W^{{\wt s}+1}_p(\mathbb{R}))\cap {\rm C}^1((0,\wt T^+), W^{{\wt s}-2}_p(\mathbb{R})).
  \end{equation*}
Moreover, by (ii) (with $(s,\ov s)=(\ov s,\wt s)$) we also have $f\in {\rm C}^1((0,{\wt T}^+), W^{3}_p(\mathbb{R})),$ and this contradicts   the maximality of $f$.
This proves the claim (iii) and  the argument is complete.
\end{proof}

We finish the section by presenting a result  used in the proof of Theorem~\ref{MT1}.

\begin{lemma}\label{L:MBC}
Given $M>0$, there exists a constant $C=C(M)$ such that
\begin{align}\label{bvm1}
\|\Phi(f)[f]\|_{W^{s-3}_p}\leq C
\end{align} 
for all $f\in W^{s+1}_p(\R)$ with  $\|f\|_{W^s_p}\leq M$.
\end{lemma}
\begin{proof} In this proof the constants denoted by $C$ depend only on~$M$.
Given~$f\in W^{s+1}_p(\R)$ with~$\|f\|_{W^s_p}\leq M$,  we have
\begin{align}\label{L:BBB2}
\|(\kappa(f)[f])'\|_{W^{s-3}_p(\R)}=\Big\|\Big(\frac{f'}{(1+f')^{1/2}}\Big)''\Big\|_{W^{s-3}_p}\leq C \Big\| \frac{f'}{(1+f')^{1/2}} \Big\|_{W^{s-1}_p}\leq C.
\end{align}
Moreover, we infer from Lemma~\ref{L:MP0} that
\begin{align}\label{Em:1}
\|\bB(f)[f']\|_{W^{s-3}_p}\leq \|\bB(f)[f']\|_{p}\leq C
\end{align}
and it remains to prove that 
\begin{align}\label{Em:2}
\|\bB(f)[(\kappa(f)[f])']\|_{W^{s-3}_p}\leq  C.
\end{align}
To this end we note that the $L_2$-adjoint $\bB^*(f)$ of $\bB(f)$ is identified by the relation 
 \[
 \pi\bB^*(f):=-(B_{0,1}(f)+B_{1,1}(f)[f,f'\cdot]).
 \]
Therefore, given $\oo,$ $\psi\in {\rm C}^\infty_0(\R)$, we have 
\begin{equation}\label{Em:3} 
 \begin{aligned}
 |\langle\bB(f)[\oo]|\psi\rangle_{W^{s-3}_p(\R)\times W^{3-s}_{p'}(\R)}|&= |\langle \oo  |\bB^*(f)[\psi]\rangle_{W^{s-3}_p(\R)\times W^{3-s}_{p'}(\R)}|\\[1ex]
 &\leq \|\oo\|_{W^{s-3}_p}\|\bB^*(f)[\psi]\|_{W^{3-s}_{p'}}.
 \end{aligned}
 \end{equation}

We next show that 
 \begin{align}\label{bvm2}
 \|\bB^*(f)[\psi]\|_{W^{3-s}_{p'}}\leq C\|\psi\|_{W^{3-s}_{p'}}
 \end{align}
 for all $ \psi\in W^{3-s}_{p'}(\R)$ and $f\in W^{s+1}_p(\R)$ with~$\|f\|_{W^s_p}\leq M$. 
 Indeed, Lemma~\ref{L:MP0} (with~$p=p'$), yields
 \begin{align}\label{bvm3}
 \|\bB^*(f)[\psi]\|_{ p' }\leq C  \|\psi\|_{p'}.
 \end{align}
Additionally,  we may argue   as in the proof of \cite[Lemma 3.5]{MBV19} to infer in view of Lemma~\ref{L:MP0}  that ${\bB^*(f)[\psi]\in W^1_{p'}(\R)}$  with
\begin{align*}
\pi( \bB^*(f)[\psi])'(x)&=\pi\bB^*(f)[\psi'](x)-\PV\int_\R \p_x\Big(\frac{y+f'(x-y)\delta_{[x,y]}f}{y^2+(\delta_{[x,y]} f)^2]}\Big)   \psi(x-y) \, dy\\[1ex]
&=\pi\bB^*(f)[\psi'](x)-\PV\int_\R \p_y\Big(\frac{ \delta_{[x,y]}f\delta_{[x,y]}f'}{y^2+(\delta_{[x,y]} f)^2]}\Big)   \psi(x-y) \, dy\\[1ex]
&=-\pi\bB(f)[\psi'](x).
\end{align*}
Invoking Lemma~\ref{L:MP2} (with $r=2-s\in(0,1-1/p)$), we get that 
\begin{align}\label{bvm4}
\|B_{0,1}^0(f)[\psi']\|_{W^{2-s}_{p'}}+\|B_{1,1}^0(f)[\psi']\|_{W^{2-s}_{p'}}\leq C\|\psi\|_{W^{3-s}_{p'}}, 
\end{align}
and since $f'\in W^{s-1}_p(\R)$ is a pointwise multiplier for $W^{2-s}_{p'}(\R)$, cf. \eqref{mult}, 
we may conclude from~\eqref{bvm3} and \eqref{bvm4} that
\begin{align*}
 \|\bB^*(f)[\psi]\|_{W^{3-s}_{p'}}\leq  \|\bB^*(f)[\psi]\|_{p'}+\|(\bB^*(f)[\psi])'\|_{W^{2-s}_{p'}}\leq C\|\psi\|_{W^{3-s}_{p'}}.
\end{align*}
This proves \eqref{bvm2}.
 Combining \eqref{Em:3} and \eqref{bvm2}, a standard density argument leads us to
 \[
 \|\bB(f)[\oo]\|_{W^{s-3}_p}\leq C\|\oo\|_{W^{s-3}_p}, \qquad \oo\in W^{s-3}_p(\R), 
 \]
 and \eqref{Em:2} follows via \eqref{L:BBB2}.
Recalling the definition \eqref{PHI} of $\Phi$, the bound~\eqref{bvm1} is a straightforward consequence of~\eqref{Em:1} and~\eqref{Em:2}.
\end{proof}

%%%%%%%%%%%%%%%%%%%%%%%%%%%%%%%%%%%%%%%%%%%%%%%%%%
%%%%%%%%%%%%%%%%%%%%%%%%%%%%%%%%%%%%%%%%%%%%%%%%%%
%%%%%%%%%%%%%%%%%%%%%%%%%%%%%%%%%%%%%%%%%%%%%%%%%%
%%%%%%%%%%%%%%%%%%%%%%%%%%%%%%%%%%%%%%%%%%%%%%%%%%
\appendix
 \section{Freezing the kernels of singular operators in Sobolev spaces with negative exponent}\label{Sec:A}
%%%%%%%%%%%%%%%%%%%%%%%%%%%%%%%%%%%%%%%%%%%%%%%%%%
%%%%%%%%%%%%%%%%%%%%%%%%%%%%%%%%%%%%%%%%%%%%%%%%%%
%%%%%%%%%%%%%%%%%%%%%%%%%%%%%%%%%%%%%%%%%%%%%%%%%%
%%%%%%%%%%%%%%%%%%%%%%%%%%%%%%%%%%%%%%%%%%%%%%%%%%

In this section we establish some technical results which enable us to locally approximate the Fr\'echet derivative $\p\Phi(\tau f)$, $\tau\in[0,1]$ and $f\in W^s_p(\R)$, 
by certain explicit Fourier multipliers, cf. Theorem~\ref{T:AP}.
These results can be viewed as a generalization of the method of freezing the coefficients of elliptic differential operators.
We extend  this method to a  particular class of  singular integral operators, one of the main difficulties arising from the fact that the Sobolev space
 where the error is measured has negative exponent. 
As a first result we establish a commutator type estimate. 

\begin{lemma}\label{L:AL1}
Let $p\in(1,2]$, $s\in(1+1/p,2),$ $\rho\in(0,(s-1-1/p)/2)$, $n,\, m\in\N,$    ${f\in W^s_p(\R)}$,  and $\varphi\in{\rm C}^1(\R)$ with uniformly continuous derivative $\varphi'$  be given.
Then, there exists a constant~$K=K(n, m, s,p,\rho, \|\varphi\|_{{\rm C}^1},   \|f\|_{W^s_p})$ such that
\begin{align} 
\| \varphi B_{n,m}^0(f)[\oo]-B_{n,m}^0(f)[\varphi \oo]\|_{W^{s-2}_{p}}\leq K \|\oo\|_{W^{ s-2-\rho}_{p}}\quad\text{for all  $\oo\in W^{s-2}_{p}(\R).$}\label{Comm1}
\end{align}

 \end{lemma}
\begin{proof}
Let $\ov\rho:=(3-s)\rho/(1-\rho)$.
 We point out that $\ov\rho\in(\rho,s-1-1/p)$ and that $\varphi\in{\rm C}^1(\R)$ is a pointwise multiplier for  $W^{ s-2-\ov\rho}_p(\R)$.
Lemma~\ref{L:MP3} (with~${s=s-\ov\rho}$) then yields 
\begin{align} 
\| \varphi B_{n,m}^0(f)[\oo]-B_{n,m}^0(f)[\varphi \oo]\|_{W^{ s-2-\ov\rho}_{p}}\leq K \|\oo\|_{W^{  s-2-\ov\rho}_{p}},\qquad \oo\in W^{  s-2-\ov\rho}_{p}(\R). \label{FEFa}
\end{align}
Besides, in view of \cite[Lemma 4.4]{AbMa20x}, we also have
\begin{align} 
\| \varphi B_{n,m}^0(f)[\oo]-B_{n,m}^0(f)[\varphi \oo]\|_{W^1_p}\leq K \|\oo\|_{p}, \qquad \oo\in L_p(\R). \label{FEFb}
\end{align}
Noticing that $W^{(1-\rho)(  s-2-\ov\rho )}_p(\R)=W^{ s-2-\ov\rho}_p(\R) $ and $W^{(1-\rho)( s-2-\ov\rho)+\rho}_p(\R)= W^{s-2}_p(\R)$, the interpolation property \eqref{IP}
 together with \eqref{FEFa} and \eqref{FEFb} leads us to the desired claim.
\end{proof}

The first important result of this section is the following lemma. 
\begin{lemma}\label{L:AL3} 
Let $n,\, m \in \N$, $p\in(1,2],$ $s\in(1+1/p,2)$, $\rho\in(0,(s-1-1/p)/2)$, ${\nu\in(0,\infty)}$, 
and~$f\in W^s_p(\R)$ be given.
For sufficiently small $\e\in(0,1)$   there exists a constant $K$ that depends on $\e,$ $ n, $ $ m,$ $  s, $ $ p,$ $ \rho, $ and $\|f\|_{W^s_p}$  such that
 \begin{equation}\label{LC1}
\Big\|\pi_j^\e B_{n,m}^0(f)[ \oo]-\frac{(f'(x_j^\e))^n}{[1+(f'(x_j^\e))^2]^m}B_{0,0}[\pi_j^\e \oo]\Big\|_{W^{ s-2}_p}\leq \nu \|\pi_j^\e \oo\|_{W^{ s-2}_p}+K\| \oo\|_{W^{s-2-\rho}_p}
 \end{equation}
for all $|j|\leq N-1$ and  $\oo\in W^{s-2}_p(\R)$.
\end{lemma} 
\begin{proof}
Taking advantage of  the relation $\chi_j^\e\pi_j^\e=\pi_j^\e$,  we have
\begin{align*}
\pi_j^\e B_{n,m}^0(f)[\oo]-\frac{(f'(x_j^\e))^n}{[1+(f'(x_j^\e))^2]^m}B_{0,0}[\pi_j^\e \oo]=T_{a}+ T_{b},
\end{align*}
 where 
 \begin{align*}
T_{a}&:=\frac{(f'(x_j^\e))^n}{[1+(f'(x_j^\e))^2]^m}(\chi_j^\e B_{0,0}[\pi_j^\e \oo]- B_{0,0}[\chi_j^\e(\pi_j^\e \oo)])- \chi_j^\e (B_{n,m}^0 (f)[\pi_j^\e \oo]-    \pi_j^\e B_{n,m}^0 (f)[ \oo  ]),\\[1ex]
T_{b}&:=  \chi_j^\e B_{n,m}^0 (f)[\pi_j^\e \oo]-\frac{(f'(x_j^\e))^n}{[1+(f'(x_j^\e))^2]^m}\chi_j^\e B_{0,0}[\pi_j^\e \oo].
\end{align*}
 Since $\pi_j^\e$ is a pointwise multiplier for $W^{s-2-\rho}_p(\R)$ and $\chi_j^\e$ is a pointwise multiplier for~$W^{s-2}_p(\R)$, Lemma \ref{L:AL1}   leads to
\begin{equation}\label{C2a}
\|T_{a}\|_{W^{ s-2}_p}\leq K\|\oo\|_{W^{ s-2-\rho}_p}.
\end{equation}

It remains to estimate $T_{b}.$ 
Using \eqref{spr3} and   the relation  ${f'(x_j^\e)=\delta_{[x,y]} (f'(x_j^\e)\id_{\R})/y},$ we~get
\begin{align*}
T_{b}&=\sum_{k=0}^{n-1}(f'(x_j^\e))^{n-k-1}\chi_j^\e B_{k+1,m}(f,\ldots,f)[f,\ldots,f, f-f'(x_j^\e)\id_{\R},\pi_j^\e \oo]\\[1ex]
&\hspace{0,45cm}- \sum_{k=0}^{m-1}\frac{(f'(x_j^\e))^{n}}{[1+ ( f'(x_j^\e))^2]^{m-k}}\chi_j^\e B_{2,k+1}(f,\ldots,f)[f-f'(x_j^\e)\id_{\R}, f+f'(x_j^\e)\id_{\R},\pi_j^\e \oo].
\end{align*} 
If $\e$ is sufficiently small,  Lemma~\ref{L:AL2} (with $\wt s=s-\rho$) implies that
\begin{align}\label{C2b}
\|T_b\|_{W^{s-2}_p}\leq \nu \|\pi_j^\e \oo\|_{W^{ s-2}_p}+K\| \oo\|_{W^{ s-2-\rho}_p}.
\end{align}
The estimate \eqref{LC1} follows from \eqref{C2a} and  \eqref{C2b}.
\end{proof}

In Lemma~\ref{L:MOL} we gather some classical properties of mollifiers.
\begin{lemma}\label{L:MOL}
Let $\eta_\delta(x):=\delta^{-1}\eta(x/\delta)$, $x\in\R$ and $\delta>0$, where $\eta\in {\rm C}^\infty_0(\R) $ is a nonnegative function with
\[
\int_\R\eta(x)\, dx=1.
\]
Given  $\delta>0$ and $f\in L_p(\R)$, with $p\in[1,\infty)$,  let $f_\delta:=f*\eta_\delta$.
The following properties hold true. 
\begin{itemize}
\item[(i)] Given $r\geq0,$ it holds $\sup_{\delta>0} \big\|[f\mapsto f_\delta]\|_{\kL(W^r_p(\R))}= 1;$
\item[(ii)]  Given $r\geq0$, there exists a constant $C>0$  such that 
\begin{align}
&\text{$\|f_\delta\|_{W^{r}_p}\leq C\delta^{r'-r}\|f\|_{W^{r'}_p}$\qquad  for all ${f\in W^{r'}_p(\R)}$, $r'\in[0,r]$, $\delta\in(0,1]$.}\label{MO1}
\end{align}
\item[(iii)] There exists a constant $C>0$  such that 
\begin{align}
&\text{$\|f_\delta-f\|_{W^{r'}_p}\leq C\delta^{r-r'}\|f\|_{W^r_p}$\qquad for all ${f\in W^r_p(\R)}$, $0\leq r'\leq r\leq1$, $\delta>0$.} \label{MO2}
\end{align}
\end{itemize}
\end{lemma}\medskip

Lemma~\ref{L:AL2} below provides the key estimate in the proof of Lemma~\ref{L:AL3}.

\begin{lemma}\label{L:AL2} 
Let $n,\, m \in \N$, $n\geq 1$, $p\in(1,2],$ $ 1+1/p<\wt s<s<2$,  $\nu\in(0,\infty)$, 
and~${f\in W^s_p(\R)}$ be given.
For sufficiently small $\e\in(0,1)$ there exists a positive constant $K=K(\e,\, n,\, m,\, s,\, \wt s, \|f\|_{W^s_p})$  such that
 \begin{equation}\label{LCw}
\|\chi_j^\e B_{n,m}(f,\ldots,f)[f,\ldots,f, f-f'(x_j^\e)\id_{\R}, \pi_j^\e \oo]\|_{W^{ s-2}_p}\leq \nu \|\pi_j^\e \oo\|_{W^{ s-2}_p}+K\| \oo\|_{W^{ \wt s-2}_p}
 \end{equation}
for all $|j|\leq N-1$ and  $\oo\in W^{s-2}_p(\R)$.
\end{lemma} 
\begin{proof}
Given $\varphi,\, \oo\in {\rm C}^\infty_0(\R)$, it holds  in view of $\chi_j^\e\pi_j^\e=\pi_j^\e$  and of the fact that the $L_2$-adjoint of $B_{n,m}^0(f) $  
 is the operator $-B_{n,m}^0(f) $, that
 \begin{align*}
&\hspace{-0.5cm}\langle\chi_j^\e B_{n,m}(f,\ldots,f)[f,\ldots,f, f-f'(x_j^\e)\id_{\R}, \pi_j^\e \oo]|\varphi\rangle_{W^{s-2}_p(\R)\times W^{2-s}_{p'}(\R)}\\[1ex]
&=-\langle  \pi_j^\e \oo|\chi_j^\e B_{n,m}(f,\ldots,f)[f,\ldots, f ,f-f'(x_j^\e)\id_{\R},\chi_j^\e\varphi]\rangle_{W^{s-2}_p(\R)\times W^{2-s}_{p'}(\R)}.
 \end{align*}
 Let $T_{j,\e}[\varphi]:=\chi_j^\e B_{n,m}(f,\ldots,f)[f,\ldots, f ,f-f'(x_j^\e)\id_{\R},\chi_j^\e\varphi]$. 
 We decompose below
 \begin{align}\label{Eq:DES}
  T_{j,\e}[\varphi] =T_{j,\e}^1[\varphi] +T^2_{j,\e}[\varphi],
 \end{align}
and we prove subsequently  that, if $\e$ sufficiently small, then
 \begin{align}\label{Eq:EST}
 \|T_{j,\e}^1[\varphi]\| _{W^{2-s}_{p'}}\leq \nu\|\varphi\|_{W^{2-s}_{p'}}\qquad\text{and}\qquad \|T_{j,\e}^2[\varphi]\| _{W^{2-\wt s}_{p'}}\leq K \|\varphi\|_{W^{2-s}_{p'}}
 \end{align}
 for all $\varphi\in W^{2-s}_{p'}(\R).$
Having established \eqref{Eq:EST}, we get
  \begin{align*}
|\langle  \pi_j^\e \oo|T_{j,\e}[\varphi]\rangle_{W^{s-2}_p(\R)\times W^{2-s}_{p'}(\R)}|&\leq \|\pi_j^\e\oo\|_{W^{s-2}_p}\|T_{j,\e}^1[\varphi]\| _{W^{2-s}_{p'}}
+\|\pi_j^\e\oo\|_{W^{\wt s-2}_p}\|T_{j,\e}^2[\varphi]\| _{W^{2-\wt s}_{p'}}\\[1ex]
&\leq (\nu\|\pi_j^\e\oo\|_{W^{s-2}_p}+K\|\oo\| _{W^{\wt s-2}_{p}})\|\varphi\|_{W^{2-s}_{p'}},
 \end{align*}
 and the claim \eqref{LCw}  follows via a standard density argument.

In order to define the terms in \eqref{Eq:DES}, let $\{\eta_\delta\}_{\delta>0}$ be a mollifier as in Lemma~\ref{L:MOL}. We set
  \begin{align*}
 T_{j,\e}^1[\varphi]&=\chi_j^\e B_{n,m}(f,\ldots,f)[f,\ldots, f, f-f'(x_j^\e)\id_{\R},\chi_j^\e(\varphi-\varphi_\delta)],\\[1ex]
 T_{j,\e}^2[\varphi]&=  \chi_j^\e B_{n,m}^0(f)[\chi_j^\e\varphi_\delta] -f'(x_j^\e)\chi_j^\e B_{n-1,m}^0(f)[\chi_j^\e\varphi_\delta],
  \end{align*}
  where $\delta=\delta(\e)\in(0,1]$ will be fixed later on and $\varphi_\delta:=\varphi*\eta_\delta$.
  
  Since $2-\wt s\in[0,1-1/p)$,  Lemma~\ref{L:MP2} (with $r=2-\wt s$) and the fact that~${\chi_j^\e\varphi_\delta\in W^{2-\wt s}_{p'}(\R)}$ yield $T_{j,\e}^2[\varphi]\in W^{2-\wt s}_{p'}(\R),$ 
  and together with \eqref{MO1} we obtain the  estimate
  \begin{align}\label{E0}
  \|T_{j,\e}^2[\varphi]\|_{W^{2-\wt s}_{p'}}\leq K\|\chi_j^\e\varphi_\delta\|_{W^{2-\wt s}_{p'}}\leq K\delta^{\wt s -s}\|\varphi\|_{W^{2-s}_{p'}}.
  \end{align}
  It remains to estimate    $T_{j,\e}^1[\varphi].$ 
  Therefore, we first infer from Lemma~\ref{L:MP0} and \eqref{MO2} that
  \begin{align}\label{E1}
  \|T_{j,\e}^1[\varphi]\|_{p'}\leq C \|\chi_j^\e(\varphi_\delta-\varphi)\|_{ p' }\leq C\|\varphi_\delta-\varphi\|_{p'}\leq C\delta^{2-s}\|\varphi\|_{W^{2-s}_{p'}}.
  \end{align}
  
  We now consider the seminorm $[T_{j,\e}^1[\varphi]]_{W^{2-s}_{p'}}$. 
  In view of  \eqref{spr3} we   write 
\begin{align*}
T_{j,\e}^1[\varphi]-\tau_\xi T_{j,\e}^1[\varphi]=T_{1}(\xi)+T_{2}(\xi)+ T_{3}(\xi),\qquad \xi\in\R,
\end{align*}
where
\begin{align*}
T_{1}(\xi)&:=(\chi_j^\e-\tau_\xi \chi_j^\e)\tau_\xi B_{n,m}(f,\ldots,f)[f,\ldots,f, f-f'(x_j^\e)\id_{\R},\chi_j^\e(\varphi-\varphi_\delta)],\\[1ex]
T_{2}(\xi)&:=  \chi_j^\e B_{n,m}(f,\ldots,f)[f,\ldots,f, f-f'(x_j^\e)\id_{\R},\chi_j^\e(\varphi-\varphi_\delta) -\tau_\xi (\chi_j^\e(\varphi-\varphi_\delta))],\\[1ex]
T_{3}(\xi)&:= \chi_j^\e\sum_{j=1}^{n-1}B_{n,m}(f,\ldots,f)[\underset{(j-1)-{\rm times}}{\underbrace{\tau_\xi f,\ldots,\tau_\xi f}},f-\tau_\xi f,f,\ldots, f, f-f'(x_j^\e)\id_{\R},\tau_\xi (\chi_j^\e(\varphi-\varphi_\delta))]\\[1ex]
&\hspace{0,45cm}+\chi_j^\e B_{n,m}(f,\ldots,f)[\tau_\xi f,\ldots,\tau_\xi f, f-\tau_\xi f,\tau_\xi (\chi_j^\e(\varphi-\varphi_\delta))]\\[1ex]
&\hspace{0,45cm}+\chi_j^\e\sum_{j=1}^m B_{n+2,m+1}^j [\tau_\xi f,\ldots,\tau_\xi f, \tau_\xi f-f'(x_j^\e)\id_{\R},\tau_\xi f+f,\tau_\xi f-f,\tau_\xi(\chi_j^\e(\varphi-\varphi_\delta))],
\end{align*}
with
$$B_{n+2,m+1}^j:=B_{n+2,m+1}(\underset{j-{\rm times}}{\underbrace{f,\ldots,f}},\tau_\xi f,\ldots,\tau_\xi f).$$
Lemma \ref{L:MP0} (with $p=p'$) together wit \eqref{MO2} yields
\begin{equation}\label{2E1}
\|T_{1}(\xi)\|_{p'}\leq C\|\chi_j^\e-\tau_\xi\chi_j^\e\|_{\infty}\|\chi_j^\e(\varphi-\varphi_\delta)\|_{p'}\leq C\delta^{2-s} \|\chi_j^\e-\tau_\xi\chi_j^\e\|_{W^1_{p'}}\|\varphi\|_{W^{2-s}_{p'}}.
\end{equation}
With regard to $T_3(\xi)$ we fix $\rho\in(s-\wt s,\min\{2-\wt s,2(s-\wt s)\})$. 
Combining Lemma \ref{L:MP1} (with~$s=\wt s$ and $r=2-\wt s-\rho\in(0,1-1/p)$)  and \eqref{MO2}  we get
\begin{equation}\label{2E2}
\begin{aligned}
\| T_{3}(\xi)\|_{p'}&\leq C\|f-\tau_\xi f\|_{W_{p'}^{2\wt s-1 -\frac{2}{p}+\rho}}\| \chi_j^\e(\varphi-\varphi_\delta)\|_{W^{2-\wt s-\rho}_{p'}}\\[1ex]
&\leq K\delta^{\rho+\wt s-s}\|f-\tau_\xi f\|_{W_{p'}^{2s-1-2/p}}\|\varphi\|_{W^{2-s}_{p'}}.
\end{aligned}
\end{equation}
In remains to estimate  the function $T_{2}(\xi),$ $\xi\in\R.$ 
To this end we  choose $a_{j,\e}\in\R$  such that ${\supp\chi_j^\e\subset[a_{j,\e}-3\e/2,a_{j,\e}+3\e/2]}$ 
and denote by $F_\e$  the Lipschitz continuous function  defined by $F_\e=f$ on ${[a_{j,\e}-2\e,a_{j,\e}+2\e]}$ and $F_\e'=f'(x_j^\e)$ on $\R\setminus[a_{j,\e}-2\e,a_{j,\e}+2\e]$.
Given~$\xi\in\R$ with $|\xi|\geq\e/2,$  it follows from Lemma~\ref{L:MP0} and~\eqref{MO2} that
\begin{align}\label{2E21}
\|T_{2}(\xi)\|_{p'}\leq   C\|\varphi-\varphi_\delta\|_{p'}\leq C\delta^{2-s}\|\varphi\|_{W^{2-s}_{p'}} .
\end{align}
If $|\xi|<\e/2$, then $\xi+ \supp\chi_j^\e\subset[a_{j,\e}-2\e,a_{j,\e}+2\e]$, and, since $f'\in {\rm C}^{s-1-1/p}(\R)$, Lemma~\ref{L:MP0}  and the properties defining~$F_\e$ lead to
\begin{equation}\label{2E22}
\begin{aligned}
\|T_{2}(\xi)\|_{p'}&=  \| \chi_j^\e B_{n,m}(f,\ldots,f)[f,\ldots,f, F_\e-f'(x_j^\e)\id_{\R},\chi_j^\e(\varphi-\varphi_\delta) -\tau_\xi (\chi_j^\e(\varphi-\varphi_\delta))]\|_{p'}\\[1ex]
&\leq C\|f'-f'(x_j^\e)\|_{L_\infty((a_{j,\e}-2\e,a_{j,\e}+2\e))}\|\chi_j^\e(\varphi-\varphi_\delta) -\tau_\xi (\chi_j^\e(\varphi-\varphi_\delta))\|_{p'}\\[1ex]
&\leq \frac{\nu}{24  }\|\chi_j^\e(\varphi-\varphi_\delta) -\tau_\xi (\chi_j^\e(\varphi-\varphi_\delta))\|_{p'},
\end{aligned}
\end{equation}
provided that $\e$ is sufficiently small. 

Combining \eqref{2E1}-\eqref{2E22}, we conclude that if $\e$ is sufficiently small, then
 \begin{equation}\label{E2}
 \begin{aligned}
{[T_{j,\e}^1[\varphi]]_{W^{2-s}_{p'}}}&\leq 3\Big(\sum_{j=1}^3\int_{\R}\frac{\|T_i(\xi)\|_{p'}^{p'}}{|\xi|^{1+(2-s)p'}}\, d\xi\Big)^{1/p'}\\[1ex]
&\leq K \delta^{\rho+\wt s-s} \|\varphi\|_{W^{2-s}_{p'}}(1+\|\chi_j^\e\|_{W^{3-s}_{p'}}+\|f\|_{W^{s}_p})+\frac{\nu}{8}\|\chi_j^\e(\varphi-\varphi_\delta)\|_{W^{2-s}_{p'}}.
 \end{aligned}
 \end{equation}
Invoking  \eqref{MES} and Lemma~\ref{L:MOL}~(i), the estimates  \eqref{E1} and \eqref{E2}  lead us to
 \begin{align*}
 \|T_{j,\e}^1[\varphi]\|_{W^{2-s}_{p'}}&\leq \frac{\nu}{2}\|\varphi\|_{W^{2-s}_{p'}}+K\delta^{\rho+\wt s-s}\|\varphi\|_{W^{2-s}_{p'}}.
 \end{align*}
 We now chose $\delta=\delta(\e)\in(0,1]$ sufficiently small to ensure that $K\delta^{\rho+\wt s-s}\leq  \nu/2.$
This choice together with \eqref{E0} shows that the estimates in \eqref{Eq:EST} hold true and the proof is complete.
\end{proof}
\pagebreak

The next lemma is the second main result of this section and describes how to freeze the kernels at infinity. 
\begin{lemma}\label{L:AL4} 
Let $n,\, m \in \N$, $p\in(1,2],$ $s\in(1+1/p,2)$, $\rho\in(0,(s-1-1/p)/2)$, ${\nu\in(0,\infty)}$, 
and~$f\in W^s_p(\R)$ be given.
For sufficiently small $\e\in(0,1)$   there exists a constant $K$ that depends on $\e,$ $ n, $ $ m,$ $  s, $ $ p,$ $ \rho, $ and $\|f\|_{W^s_p}$  such that
  \begin{equation}\label{LCN0}
\Big\|\pi_N^\e B_{0,m}^0(f)[ \oo]-B_{0,0}[\pi_N^\e \oo]\Big\|_{W^{ s-2}_p}\leq \nu \|\pi_N^\e \oo\|_{W^{ s-2}_p}+K\| \oo\|_{W^{   s-2-\rho}_p}
 \end{equation}
 and
 \begin{equation}\label{LCN}
\|\pi_N^\e B_{n,m}^0(f)[ \oo]\|_{W^{ s-2}_p}\leq \nu \|\pi_N^\e \oo\|_{W^{ s-2}_p}+K\| \oo\|_{W^{s-2-\rho}_p},\quad n\geq 1,
 \end{equation}
for all $\oo\in W^{s-2}_p(\R)$.
\end{lemma} 
\begin{proof}
Similarly as in the proof of Lemma~\ref{L:AL3} we write
\begin{align*}
\pi_N^\e B_{0,m}^0(f)[\oo]-B_{0,0}[\pi_N^\e \oo]=T_{a}+ T_{b},
\end{align*}
 where 
 \begin{align*}
T_{a}&:=(\chi_N^\e B_{0,0}[\pi_N^\e \oo]- B_{0,0}[\chi_N^\e(\pi_N^\e \oo)])-\chi_N^\e( B_{0,m}^0 (f)[\pi_N^\e \oo]-  \pi_N^\e B_{0,m}^0 (f)[ \oo  ]),\\[1ex]
T_{b}&:=  \chi_N^\e( B_{0,m}^0 (f)[\pi_N^\e \oo]-  B_{0,0}[\pi_N^\e \oo]).
\end{align*}
 Since $\pi_j^\e$ is a pointwise multiplier for $W^{s-2-\rho}_p(\R)$ and $\chi_j^\e$ is a pointwise multiplier for~$W^{s-2}_p(\R)$, we get in view of Lemma~\ref{L:AL1} that
\begin{equation}\label{C2aN}
\|T_{a}\|_{W^{ s-2}_p}\leq K\|\oo\|_{W^{ s-2-\rho}_p}.
\end{equation}
With respect to $T_b$ we infer from \eqref{spr3} that
\begin{align*}
T_{b}&=- \chi_N^\e\sum_{i=1}^{m} B_{2,i}^0(f)[\pi_N^\e \oo].
\end{align*} 
If $\e$ sufficiently small, it follows from Lemma~\ref{L:AL5} that
\begin{align}\label{C2bN}
\|T_b\|_{W^{s-2}_p}\leq \nu \|\pi_N^\e \oo\|_{W^{ s-2}_p}+K\| \oo\|_{W^{  s-2-\rho}_p}.
\end{align}
The claim \eqref{LCN0} follows now from \eqref{C2aN} and  \eqref{C2bN}. 
Finally, the assertion \eqref{LCN} is obtained  by arguing as in the proof of \eqref{LCN0}.
\end{proof}

Lemma~\ref{L:AL5} below provides the crucial estimate in the proof of Lemma~\ref{L:AL4}.
\begin{lemma}\label{L:AL5} 
Let $n,\, m \in \N$, $n\geq 1$, $p\in(1,2],$ $ 1+1/p<\wt s<s<2$,  $\nu\in(0,\infty)$, 
and~${f\in W^s_p(\R)}$ be given.
For sufficiently small $\e\in(0,1)$, there exists a positive constant $K=K(\e,\, n,\, m,\, s,\,\wt s,\,\|f\|_{W^s_p})$  such that
 \begin{equation}\label{LCwN}
\|\chi_N^\e B_{n,m}^0(f)[\pi_N^\e \oo]\|_{W^{ s-2}_p}\leq \nu \|\pi_N^\e \oo\|_{W^{ s-2}_p}+K\| \oo\|_{W^{ \wt s-2}_p}
 \end{equation}
for all    $\oo\in W^{s-2}_p(\R)$.
\end{lemma} 
\begin{proof}
Given $\varphi,\, \oo\in {\rm C}^\infty_0(\R)$, it  follows by arguing as in the proof of Lemma~\ref{L:AL2} that
 \begin{align*}
\langle\chi_N^\e B_{n,m}^0(f)[\pi_N^\e \oo]|\varphi\rangle_{W^{s-2}_p(\R)\times W^{2-s}_{p'}(\R)}
=-\langle  \pi_N^\e \oo|\chi_N^\e B_{n,m}^0(f)[\chi_N^\e\varphi]\rangle_{W^{s-2}_p(\R)\times W^{2-s}_{p'}(\R)}.
 \end{align*}
We write $\chi_N^\e B_{n,m}^0(f)[\chi_N^\e\varphi]=T_{\e}^1[\varphi]+T^2_{\e}[\varphi]$
where, for  sufficiently small $\e$, we have
 \begin{align}\label{Eq:ESTN}
 \|T_{\e}^1[\varphi]\| _{W^{2-s}_{p'}}\leq \nu\|\varphi\|_{W^{2-s}_{p'}}\qquad\text{and}\qquad \|T_{\e}^2[\varphi]\| _{W^{2-\wt s}_{p'}}\leq K \|\varphi\|_{W^{2-s}_{p'}}
 \end{align}
 for all $\varphi\in W^{2-s}_{p'}(\R).$
The estimates in \eqref{Eq:ESTN} together with the previous identity imply
  \begin{align*}
&\hspace{-0.5cm}|\langle\chi_N^\e B_{n,m}^0(f)[\pi_N^\e \oo]|\varphi\rangle_{W^{s-2}_p(\R)\times W^{2-s}_{p'}(\R)}|\\[1ex]
&\leq \|\pi_N^\e\oo\|_{W^{s-2}_p}\|T_{\e}^1[\varphi]\| _{W^{2-s}_{p'}}
+\|\pi_N^\e\oo\|_{W^{\wt s-2}_p}\|T_{\e}^2[\varphi]\| _{W^{2-\wt s}_{p'}}\\[1ex]
&\leq (\nu\|\pi_N^\e\oo\|_{W^{s-2}_p}+K\|\oo\| _{W^{\wt s-2}_{p}})\|\varphi\|_{W^{2-s}_{p'}},
 \end{align*}
 and \eqref{LCwN} follows.

Let again  $\{\eta_\delta\}_{\delta>0}$ be a mollifier and set 
  \begin{align*}
 T_{\e}^1[\varphi]=\chi_N^\e B_{n,m}^0(f)[\chi_N^\e(\varphi-\varphi_\delta)]\qquad\text{and}\qquad T_{\e}^2[\varphi]=  \chi_N^\e B_{n,m}^0(f)[\chi_N^\e\varphi_\delta],
  \end{align*}
  where $\varphi_\delta:=\varphi*\eta_\delta$ and with $\delta=\delta(\e)\in(0,1]$ which we fix later on.
 Taking advantage of Lemma~\ref{L:MP2} (with~${r=2-\wt s\in(0,1-1/p)}$) and of the fact that  $\chi_N^\e\varphi_\delta\in W^{2-\wt s}_{p'}(\R)$, we conclude that   $T_{\e}^2[\varphi]\in W^{2-\wt s}_{p'}(\R)$ 
  and together with \eqref{MO1} we obtain the   estimate
  \begin{align}\label{E0N}
  \|T_{\e}^2[\varphi]\|_{W^{2-\wt s}_{p'}}\leq K \|\chi_N^\e\varphi_\delta\|_{W^{2-\wt s}_{p'}}\leq K\delta^{\wt s -s}\|\varphi\|_{W^{2-s}_{p'}}.
  \end{align}
  We now estimate  $T_{\e}^1[\varphi].$ 
 Combining Lemma~\ref{L:MP0} and  \eqref{MO2}, we have
  \begin{align}\label{E1N}
  \|T_{\e}^1[\varphi]\|_{p'}\leq C \|\chi_N^\e(\varphi_\delta-\varphi)\|_{ p' }\leq C\|\varphi-\varphi_\delta\|_{p'}\leq C\delta^{2-s}\|\varphi\|_{W^{2-s}_{p'}},
  \end{align}
and it remains to consider  the seminorm $[T_{\e}^1[\varphi]]_{W^{2-s}_{p'}}$. 
Using \eqref{spr3}, we get  
\begin{align*}
T_{\e}^1[\varphi]-\tau_\xi T_{\e}^1[\varphi]=T_{1}(\xi)+T_{2}(\xi)+ T_{3}(\xi),\qquad \xi\in\R,
\end{align*}
where 
\begin{align*}
T_{1}(\xi)&:=(\chi_N^\e-\tau_\xi \chi_N^\e)\tau_\xi (B_{n,m}^0(f)[\chi_N^\e(\varphi-\varphi_\delta)]),\\[1ex]
T_{2}(\xi)&:=  \chi_N^\e B_{n,m}^0(f)[ \chi_N^\e(\varphi-\varphi_\delta) -\tau_\xi (\chi_N^\e(\varphi-\varphi_\delta))],\\[1ex]
T_{3}(\xi)&:= \chi_N^\e\sum_{j=1}^{n}B_{n,m}(f,\ldots,f)[\underset{(j-1)-{\rm times}}{\underbrace{\tau_\xi f,\ldots,\tau_\xi f}},f-\tau_\xi f,f,\ldots, f,\tau_\xi (\chi_N^\e(\varphi-\varphi_\delta))]\\[1ex]
&\hspace{0,45cm}+\chi_N^\e\sum_{j=1}^m B_{n+2,m+1}^j [\tau_\xi f,\ldots,\tau_\xi f,\tau_\xi f+f,\tau_\xi f-f,\tau_\xi(\chi_N^\e(\varphi-\varphi_\delta))],
\end{align*}
with $B_{n+2,m+1}^j$ as defined in the proof of Lemma~\ref{L:AL2}.
Arguing as in Lemma~\ref{L:AL2}, we obtain
\begin{equation}\label{2E1N}
\|T_{1}(\xi)\|_{p'}\leq  C\delta^{2-s} \|\chi_N^\e-\tau_\xi\chi_N^\e\|_{W^1_{p'}}\|\varphi\|_{W^{2-s}_{p'}}
\end{equation}
and, for some fixed $\rho\in(s-\wt s,\min\{2-\wt s,2(s-\wt s)\})$,
\begin{equation}\label{2E2N}
\| T_{3}(\xi)\|_{p'}\leq K\delta^{\rho+\wt s-s}\|f-\tau_\xi f\|_{W_{p'}^{2s-1-2/p}}\|\varphi\|_{W^{2-s}_{p'}}.
\end{equation}

In order to estimate $T_{2}(\xi),$ let $F_\e$ denote the Lipschitz continuous function equal to~$f$ on the set $\{|x|\geq 1/\e-2\e\}$ and  which is linear in the interval $\{|x|\leq 1/\e-2\e\}$.
If $|\xi|\geq\e,$  Lemma~\ref{L:MP0} and~\eqref{MO2} yield
\begin{align}\label{2E21N}
\|T_{2}(\xi)\|_{p'}\leq   C\|\varphi-\varphi_\delta\|_{p'}\leq C\delta^{2-s}\|\varphi\|_{W^{2-s}_{p'}} .
\end{align}
If $|\xi|<\e$,  we note that   $\xi+ \supp\chi_N^\e\subset \{|x|\geq 1/\e-2\e\}$.  
Lemma \ref{L:MP0}, the definition of~$F_\e$, and the observation that
 $\|F_\e'\|_{\infty}\leq 2\e\|f\|_\infty+\|f'\|_{L_\infty(\{|x|\geq 1/\e-2\e\})}$ (which holds if $\e$ is sufficiently small), then lead to
\begin{equation}\label{2E22N}
\begin{aligned}
\|T_{2}(\xi)\|_{p'}&=  \| \chi_N^\e B_{n,m}(f,\ldots,f)[f,\ldots,f, F_\e,\chi_N^\e(\varphi-\varphi_\delta) -\tau_\xi (\chi_N^\e(\varphi-\varphi_\delta))]\|_{p'}\\[1ex]
&\leq C\|F_\e' \|_{\infty}\|\chi_N^\e(\varphi-\varphi_\delta) -\tau_\xi (\chi_N^\e(\varphi-\varphi_\delta))\|_{p'}\\[1ex]
&\leq \frac{\nu}{24  }\|\chi_N^\e(\varphi-\varphi_\delta) -\tau_\xi (\chi_N^\e(\varphi-\varphi_\delta))\|_{p'},
\end{aligned}
\end{equation}
provided that $\e$ is sufficiently small. To obtain the last inequality we have taken advantage of the fact that $f'\in W^{s-1}_p(\R)$ vanishes at infinity. 

From \eqref{2E1N}-\eqref{2E22N}, we conclude, for $\e$  sufficiently small, that
 \begin{equation}\label{E2N}
\hspace{-0.15cm}{[T_{\e}^1[\varphi]]_{W^{2-s}_{p'}}}\leq K \delta^{\rho+\wt s-s} \|\varphi\|_{W^{2-s}_{p'}}(1+\|1-\chi_N^\e\|_{W^{3-s}_{p'}}+\|f\|_{W^{s}_p})+\frac{\nu}{8}\|\chi_N^\e(\varphi-\varphi_\delta)\|_{W^{2-s}_{p'}}.
 \end{equation}
Combining \eqref{MES},  Lemma~\ref{L:MOL}~(i), \eqref{E1N}, and \eqref{E2N}  we get
 \begin{align*}
 \|T_{\e}^1[\varphi]\|_{W^{2-s}_{p'}}&\leq \frac{\nu}{2}\|\varphi\|_{W^{2-s}_{p'}}+K\delta^{\rho+\wt s-s}\|\varphi\|_{W^{2-s}_{p'}}.
 \end{align*}
 Choosing now $\delta=\delta(\e)$ sufficiently small to ensure that $K\delta^{\rho+\wt s-s}\leq  \nu/2,$
we obtain, together with  \eqref{E0N}, the desired estimates \eqref{Eq:ESTN}  and the proof is complete.
\end{proof}

%%%%%%%%%%%%%%%%%%%%%%%%%%%%%%%%%%%%%%%%%%%%%%%%%%
%%%%%%%%%%%%%%%%%%%%%%%%%%%%%%%%%%%%%%%%%%%%%%%%%%
%%%%%%%%%%%%%%%%%%%%%%%%%%%%%%%%%%%%%%%%%%%%%%%%%%
%%%%%%%%%%%%%%%%%%%%%%%%%%%%%%%%%%%%%%%%%%%%%%%%%%
 \section{Estimates for some pointwise multipliers}\label{Sec:b}
%%%%%%%%%%%%%%%%%%%%%%%%%%%%%%%%%%%%%%%%%%%%%%%%%%
%%%%%%%%%%%%%%%%%%%%%%%%%%%%%%%%%%%%%%%%%%%%%%%%%%
%%%%%%%%%%%%%%%%%%%%%%%%%%%%%%%%%%%%%%%%%%%%%%%%%%
%%%%%%%%%%%%%%%%%%%%%%%%%%%%%%%%%%%%%%%%%%%%%%%%%%
In  this appendix  we present the proofs of the estimates \eqref{EQ:Mul1}-\eqref{EQ:Mul2}. 
The estimate \eqref{EQ:Mul2} is used in the proof of Theorem ~\ref{T:AP} and  \eqref{EQ:Mul1} is an important argument when establishing~\eqref{EQ:Mul2}.

\begin{proof}[Proof of  \eqref{EQ:Mul1}]
Since $W^{r}_p(\R)\hookrightarrow {\rm C}(\R)$, we have
\begin{align}\label{Ea1}
\|gh\|_{p'}\leq \|g\|_\infty\|h\|_{p'},
\end{align}
and it remains to estimate the term
\begin{align*}
[gh]_{W^{1-r}_{p'}}^{p'}=\int_\R\frac{\|gh-\tau_\xi(gh)\|_{p'}^{p'}}{|\xi|^{1+(1-r)p'}}\, d\xi\leq 2^{p'}\Big(\|g\|_\infty^{p'}\int_\R\frac{\|h-\tau_\xi h\|_{p'}^{p'}}{|\xi|^{1+(1-r)p'}}\, d\xi
+ \int_\R\frac{\|(g-\tau_\xi g)\tau_\xi h\|_{p'}^{p'}}{|\xi|^{1+(1-r)p'}}\, d\xi\Big).
\end{align*}
According to \cite[Theorem~4.1]{Am91}, the multiplication
\[
[(g,h)\mapsto gh]: W^{2(r-1/p)}_{p'}(\R)\times W^{1-r-\rho}_{p'}(\R)\to L_{p'}(\R) 
\]
is continuous, hence
\begin{align*}
[gh]_{W^{1-r}_{p'}}^{p'}\leq 2^{p'}\|g\|_\infty^{p'}\|h\|_{W^{1-r}_{p'}}^{p'}+C\|h\|_{W^{1-r-\rho}_{p'}}^{p'}\int_\R\frac{\|g-\tau_\xi g\|_{W^{2(r-1/p)}_{p'}}^{p'}}{|\xi|^{1+(1-r)p'}}\, d\xi.
\end{align*} 
  Lemma~\ref{L:AL0} (with $t'=2(r-1/p)$ and $t=r+1-2/p$) yields
\begin{align*}
\int_\R\frac{\|g-\tau_\xi g\|_{W^{2(r-1/p)}_{p'}}^{p'}}{|\xi|^{1+(1-r)p'}}\, d\xi \leq \|g\|_{W^{r+1-2/p}_{p'}}^{p'}\leq C \|g\|_{W^{r}_p}^{p'}.
\end{align*} 
We thus conclude that
\begin{align}\label{Ea2}
[gh]_{W^{1-r}_{p'}}^{p'}\leq2^{p'}\|g\|_\infty^{p'}\|h\|_{W^{1-r}_{p'}}^{p'}+C \|g\|_{W^{r}_p}^{p'}\|h\|_{W^{1-r-\rho}_{p'}}^{p'},
\end{align}
and \eqref{Ea1} together with \eqref{Ea2} lead to the desired claim.
\end{proof}

We are now in a position to prove  \eqref{EQ:Mul2}.

\begin{proof}[Proof of \eqref{EQ:Mul2}]
Given $h,\, \psi\in {\rm C}^\infty_0(\R)$, we have
\[
\langle \varphi h|\psi\rangle_{W^{r-1}_p(\R)\times W^{1-r}_{p'}(\R)}=\langle  h|\varphi \psi\rangle_{W^{r-1}_p(\R)\times W^{1-r}_{p'}(\R)}
=\langle  h|T_1[\psi]+T_2[\psi]\rangle_{W^{r-1}_p(\R)\times W^{1-r}_{p'}(\R)},
\]
where $T_1[\psi]\in W^{1-r}_{p'}(\R)$ and $ T_2[\psi]\in W^{1- r+\rho}_{p'}(\R)$ are defined in~\eqref{tipsi} and satisfy
\begin{align}\label{TD}
\|T_1[\psi]\|_{W^{1-r}_{p'}}\leq 5\|\varphi\|_\infty\|\psi\|_{W^{1-r}_{p'}}\quad\text{and}\quad 
\|T_2[\psi]\|_{W^{1-r+\rho}_{p'}}\leq C\frac{ \|\varphi\|_{W^{r}_{p}}^{1+2r/\rho}}{\|\varphi\|_{\infty}^{2r/\rho}}\|\psi\|_{W^{1-r}_{p'}}.
\end{align}
Having established \eqref{TD}, we get 
\begin{align*}
|\langle \varphi h|\psi\rangle|_{W^{r-1}_p(\R)\times W^{1-r}_{p'}(\R)}&\leq \| h\|_{W^{r-1}_p}\|T_1[\psi]\|_{W^{1-r}_{p'}}+\| h\|_{W^{ r-1-\rho}_p}\|T_2[\psi]\|_{W^{1- r+\rho}_{p'}}\\[1ex]
&\leq  \Big(5\|\varphi\|_\infty \|h\|_{W^{r-1}_p}+C\frac{ \|\varphi\|_{W^{r}_{p}}^{1+2r/\rho}}{\|\varphi\|_{\infty}^{2r/\rho}}\|h\|_{W^{r-1-\rho}_p}\Big)\|\psi\|_{W^{1-r}_{p'}},
\end{align*}
and the estimate \eqref{EQ:Mul2}  follows.
The functions in \eqref{TD} are defined by
\begin{align}\label{tipsi}
T_1[\psi]:=\varphi(\psi-\psi_\delta)\quad\text{and}\quad   T_2[\psi]:=\varphi \psi_\delta,
\end{align}
where $\{\eta_\delta\}_{\delta>0}$ is a mollifier, $\psi_\delta=\psi*\eta_\delta$, and $\delta\in(0,1]$ is chosen below.
Combining~\eqref{EQ:Mul1}, Lemma~\ref{L:MOL}~(i), and~\eqref{MO2}, we get 
\begin{align*}
\|T_1[\psi]\|_{W^{1-r}_{p'}}&\leq 2\|\varphi\|_\infty \|\psi-\psi_\delta\|_{W^{1-r}_{p'}}+C\|\varphi\|_{W^{r}_p}\|\psi-\psi_\delta\|_{W^{1-r-\rho}_{p'}} \\[1ex]
&\leq (4\|\varphi\|_\infty  +C\delta^{\rho}\|\varphi\|_{W^{r}_p})\|\psi\|_{W^{1-r}_{p'}}.
\end{align*}
After eventually choosing $C$ to be larger than the norm of the embedding $W^r_p(\R)\hookrightarrow L_\infty(\R),$ we  set $\delta= (\|\varphi\|_\infty/(C\|\varphi\|_{W^r_p}))^{1/\rho}\in(0,1]$, and obtain that
\begin{align*}
\|T_1[\psi]\|_{W^{1-r}_{p'}}& \leq 5\|\varphi\|_\infty \|\psi\|_{W^{1-r}_{p'}}.
\end{align*}

With respect to $T_2[\psi]$, we note that since $\varphi\in W^{r+1-2/p}_{p'}(\R)$ and $\psi_\delta\in W^{r+1}_{p'}(\R)$, it holds that
$T_2[\psi]\in W^{r+1-2/p}_{p'}(\R)\hookrightarrow  W^{1-r+\rho}_{p'}(\R)$, cf. \eqref{adMP}, and together with~\eqref{adMP} and~\eqref{MO1} we get
\begin{align*}
\|T_2[\psi]\|_{W^{1-r+\rho}_{p'}}\leq C\|\psi_\delta\|_{W^{r+1}_{p'}}\|\varphi\|_{W^{r+1-2/p}_{p'}}\leq C  \delta^{-2r}\|\varphi\|_{W^{r}_{p}}\|\psi\|_{W^{1-r}_{p'}}
\leq C\frac{ \|\varphi\|_{W^{r}_{p}}^{1+2r/\rho}}{\|\varphi\|_{\infty}^{2r/\rho}} \|\psi\|_{W^{1-r}_{p'}}.
\end{align*}
Hence, both estimates in \eqref{TD} hold true and the proof is complete.
\end{proof}

\section*{Acknowledgement}
The authors were partially supported by the RTG 2339 ''Interfaces, Complex Structures, and Singular Limits''
of the German Science Foundation (DFG). The support is gratefully acknowledged.

\bibliographystyle{siam}
\bibliography{AnBo}
\end{document}